\documentclass[12pt]{amsart}
\usepackage{epsfig,color,esint}

\makeatother

\theoremstyle{definition}

\def\fnum{equation}
\newtheorem{Thm}[\fnum]{Theorem}
\newtheorem{Cor}[\fnum]{Corollary}

\newtheorem{Lem}[\fnum]{Lemma}
\newtheorem{Con}[\fnum]{Conjecture}
\newtheorem{Def}[\fnum]{Definition}

\newtheorem{Pro}[\fnum]{Proposition}

\numberwithin{equation}{section}

\newcommand{\Vol}{{\text{Vol}}}

\newcommand{\nn}{{\bf{n}}}
\newcommand{\Ric}{{\text{Ric}}}

\newcommand{\Tr}{{\text{Tr}}}

\def\RR{{\bold R}}

\def\SS{{\bold S}}

\newcommand{\dv}{{\text {div}}}
\newcommand{\e}{{\text {e}}}

\newcommand{\cC}{{\mathcal{C}}}

\newcommand{\cF}{{\mathcal{F}}}
\newcommand{\cL}{{\mathcal{L}}}

\newcommand{\cW}{{\mathcal{W}}}

\newcommand{\cN}{{\mathcal{N}}}

\newcommand{\cP}{{\mathcal{P}}}

\newcommand{\cR}{{\mathcal{R}}}

\newcommand{\eqr}[1]{(\ref{#1})}

\def\bH{{\bold H}}

\title[Complexity of parabolic systems]{Complexity of parabolic systems}
\author[]{Tobias Holck Colding}%
\address{MIT, Dept. of Math.\\
77 Massachusetts Avenue, Cambridge, MA 02139-4307.}
\author[]{William P. Minicozzi II}%
 
\thanks{The  authors
were partially supported by NSF Grants DMS 1812142 and DMS 1707270.}

\email{colding@math.mit.edu  and minicozz@math.mit.edu}

\begin{document}

\begin{abstract}
We first bound the codimension of an ancient  mean curvature flow by the entropy.    As a consequence,  all blowups  lie  in a Euclidean subspace whose dimension is bounded by the entropy and dimension  of the evolving submanifolds.    This drastically reduces the complexity of the system.  Combined with \cite{CM12}, this gives the first general bounds on generic singularities of surfaces in arbitrary codimension.

We also show sharp bounds for codimension in arguably some of the most important situations of ancient flows. Namely, we prove that in any dimension and codimension any ancient flow that is  cylindrical at $-\infty$ must be a flow of hypersurfaces in a Euclidean subspace.  This extends well-known classification results to higher codimension.

The bound on the codimension in terms of the entropy is a special case of  sharp bounds  for spectral counting functions for shrinkers and, more generally, ancient flows.  Shrinkers are solutions that evolve by scaling and are the singularity models for the flow.     

Finally, we show rigidity of cylinders as shrinkers in all dimension and all codimension in a very strong sense: Any  shrinker, even in a large dimensional space, that is sufficiently close to a cylinder on a large enough, but compact, set is itself   a cylinder.       This is an important tool in the theory and is key for regularity; cf. \cite{CM8}.  
\end{abstract}

\maketitle

\section{Introduction}
We introduce a new circle of ideas that gives a new way of attacking mean curvature flow (MCF) in higher codimension.
      Higher codimension  MCF  is a   complicated nonlinear parabolic system where much less is known than for hypersurfaces.  The complexity of the system increases as the codimension increases.   
We  show that blowups of higher codimension MCF have much smaller codimension  than   the original flow.   In many important instances, we   show that blowups are evolving hypersurfaces in a  Euclidean subspace even when the original flow is far from being hypersurfaces.

One way of thinking about MCF is as a one-parameter family of submanifolds $M_t \subset \RR^N$ evolving so that the position vector $x\in M^n_t$ satisfies the nonlinear heat equation
\begin{align}
(\partial_t-\Delta_{M_t})\,x=0\, .
\end{align}
This equation is nonlinear since the Laplacian depends on the evolving submanifold $M_t$.   Many fundamental results and tools about elliptic PDEs have originated in the study of 
the minimal surface equation.  In much the same way,  MCF  is one of the most fundamental parabolic systems.  New results and tools are expected to apply to a variety of other   systems.  

There is a Lyapunov function for the flow that is particularly useful.  To define it recall that the Gaussian surface area $F$ of an $n$-dimensional submanifold
 $\Sigma^n \subset \RR^N$ is
\begin{align}
	F(\Sigma) = \left( 4\,\pi \right)^{ - \frac{n}{2}} \,  \int_{\Sigma} \e^{ - \frac{|x|^2}{4} } \, .
\end{align}
Following \cite{CM6}, the entropy $\lambda$ is the supremum of $F$ over all translations and dilations
\begin{align}
	\lambda (\Sigma) = \sup_{c,x_0} \, F (c\,\Sigma + x_0) \, .
\end{align}
 By Huisken's monotonicity, \cite{Hu}, it follows that $\lambda$ is monotone nonincreasing under the flow.  From this, and lower semi continuity of $\lambda$, 
we have that all blowups have entropy bounded by that of the initial submanifold in a MCF.   

\subsection{Liouville properties}       Let $M_t^n \subset \RR^N$ be an ancient MCF of $n$-dimensional submanifolds with entropies 
$\lambda (M_t)\leq \lambda_0<\infty$.  Ancient   flows are solutions that exist  for all negative times.   The  space  $\cP_d$ of polynomial growth caloric functions consists of   $u(x,t)$ on $\cup_t M_t\times\{t\}$ so that  $(\partial_t - \Delta_{M_t})\, u=0$ and there exists $C$ depending on $u$ with
  \begin{align}
  	|u(x,t)| \leq C\, (1 + |x|^d+|t|^{\frac{d}{2}}) {\text{ for all }} (x,t) \text{ with }x\in M_t , \, t < 0 \, .
  \end{align}
  Motivated by \cite{CM1}--\cite{CM5}, similar spaces were considered in Calle's thesis \cite{Ca1}, \cite{Ca2}.
  
  Our first theorem is a sharp   bound for a parabolic ``counting function'' on ancient MCF (in all of these results, the time slices $M_t$ are allowed to be non-compact):

  \begin{Thm}	\label{t:maria}
  There exists $C_n$  so that if $M_t^n \subset \RR^N$ is an ancient MCF with $\lambda (M_t) \leq \lambda_0$ and $d \geq 1$, then 
$ \dim \cP_d \leq C_n \, \lambda_0 \, d^{n}$.
  \end{Thm}

  The dependence on $d$ is sharp on Euclidean space, where $\cP_d (\RR^n)$ consists of the classical caloric polynomials.  
  For a fixed manifold with $\Ric \geq 0$ that is time-independent, the related bound $C\, d^{n+1}$ was proven 
    by  Lin and Zhang, \cite{LZ}, adapting the arguments of \cite{CM1}--\cite{CM5} for harmonic functions.
The  sharp bound $C \, d^n$ in that case was proven in 
  \cite{CM9}.  These  time-independent bounds use the commutativity of $\Delta$ and $\partial_t$ and do not apply here.   
    Instead a key here
   is a new localization inequality for the Gaussian $L^2$ norm. 
     This new approach allows us to obtain the optimal dependence; see \cite{CM10} for more.
  Similar localization ideas also play a role  later in this paper.

  Theorem \ref{t:maria} has a number of applications, including bounds for the associated heat kernel.  One remarkable consequence with $d=1$ is a bound for the codimension.  This is because the flow sits inside a linear subspace of dimension at most $\dim \, \cP_1$ 
  since a linear relation for coordinate functions specifies a hyperplane containing the flow.

      \begin{Cor}	\label{c:counting2}
  There exists $C_n$ so that if $M_t^n \subset \RR^N$ is an ancient MCF, then it is contained in a Euclidean subspace of dimension   $\leq C_n \, \sup_t\lambda (M_t)$.
  \end{Cor}

    Singularities are modeled by shrinkers $\Sigma$ that evolve by scaling.       The most fundamental shrinkers are cylinders $\SS^k_{\sqrt{2k}} \times \RR^{n-k}$,
but there are many others including all $n$-dimensional minimal submanifolds of the sphere $\partial B_{\sqrt{2n}} \subset \RR^N$.   
    Let  $\Sigma^n \subset \RR^N$ be a shrinker with finite entropy $\lambda (\Sigma)$. 
As in \cite{CM6}, the drift Laplacian  (Ornstein-Uhlenbeck operator) $\cL = \Delta - \frac{1}{2} \nabla_{x^T}$ is self-adjoint with respect to the Gaussian inner product
$
	\int_{\Sigma} u\, v \, \e^{ -  \frac{|x|^2}{4} }  \, .
$
 Let  $\| u \|_{L^2}$  denote the Gaussian $L^2$ norm.
We will say that $u$ is a $\mu$-eigenfunction   if $\cL \, u = - \mu \, u$   and $0 < \| u \|_{L^2} < \infty$. 
    The {\it spectral counting function} $\cN (\mu)$ is the number of eigenvalues $\mu_i\leq \mu$ counted with multiplicity.
     The next result  bounds  $\cN$:
  
   \begin{Thm}	\label{t:countingN}
 There exists $C_n$ so that the counting function for $\cL$ on an $n$-dimensional shrinker $\Sigma^n \subset \RR^N$   satisfies $\cN (\mu) \leq C_n \, \lambda (\Sigma) \, \mu^{n}$ for $\mu \geq 1$.
  \end{Thm}
    
       The dependence on $\mu$ is sharp even on Euclidean space.       A key component in the proof is a sharp polynomial growth bound for eigenfunctions of $\cL$ on any shrinker.  This result is of independent interest.  It too is sharp on $\RR^n$ and shows that any eigenfunction on any shrinker grows polynomially of degree at most twice the eigenvalue (see Theorem \ref{t:polyGu}).  
    	
	Specializing Theorem \ref{t:countingN} to $\mu = \frac{1}{2}$ gives:
	
  \begin{Cor}	\label{c:counting}
  There exists $C_n$ so that if $\Sigma^n \subset \RR^N$ is a shrinker, then it is contained in a Euclidean subspace of dimension   $\leq C_n \, \lambda(\Sigma)$.
  \end{Cor}
  
  Combined with \cite{CM12}, Corollary \ref{c:counting} shows that all closed $2$-dimensional singularities for higher codimension mean curvature flow that cannot be perturbed away have uniform entropy bounds and lie in a linear subspace of small dimension.   This gives the first general bounds for generic singularities   in higher codimension.

Our estimates in Corollaries \ref{c:counting2} and \ref{c:counting} are linear in the entropy.  The corresponding linear estimate   for algebraic varieties in
 complex projective space follows from B\'ezout's theorem; see corollary $18.12$ in \cite{Ha}.    
 When $\Sigma \subset \partial B_{\sqrt{2\,n}}\subset \RR^N$ is a closed $n$-dimensional minimal submanifold of the sphere and the entropy reduces to the volume, this estimate follows   Cheng-Li-Yau, \cite{CgLYa}.

   \subsection{Sharp bound for codimension}
   The next result gives sharp bounds for codimension in arguably some of the most important situations for ancient flows.   The bounds in the previous subsection were sharp in the 
   exponent of $d$ and, thus, asymptotically sharp as $d\to \infty$.  The next result is more delicate and obtains sharp constants for  $d$ fixed.
   
Suppose that $M_t^n\subset \RR^N$ is an ancient MCF with $\sup_t\lambda (M_t)<\infty$.  For each constant $c>0$ define the flow $M_{c,t}$ by
$
M_{c,t}=\frac{1}{c}\,M_{c^{2}\,t}
$.
It follows that $M_{c,t}$ is an ancient MCF as well.  Since  $\sup_t\lambda(M_t)<\infty$, it follows from Huisken's monotonicity, \cite{Hu}, and work of Ilmanen, \cite{I}, White, \cite{W3}, that every sequence $c_i\to \infty$ has a subsequence (also denoted by $c_i$) so that $M_{c_i,t}$ converges to a shrinker $M_{\infty,t}$ (so $M_{\infty,t}=\sqrt{-t}\,M_{\infty,-1}$) with $\sup_t\lambda (M_{\infty,t})\leq \sup_t\lambda (M_t)$.  We will say that such a $M_{\infty,t}$ is a tangent flow at $-\infty$ of the original flow.  
We next give a sharp bound for the codimension:

\begin{Thm} \label{t:ancientcylinder}
  If $M^n_t\subset\RR^N$ is an ancient MCF and one tangent flow at $-\infty$ is a cylinder $\SS^k_{\sqrt{2\,k}}\times \RR^{n-k}$, then $M_t$ is  a flow of hypersurfaces in a Euclidean subspace.
  \end{Thm}

 \vskip1mm
  We believe that Theorem \ref{t:ancientcylinder}  will have wide ranging consequences for MCF in higher codimension.   We will try here to briefly explain some of these  (more discussion is in the subsection ``Further applications'' at the end of the introduction).  
  
  Using   Angenent-Daskalopoulos-Sesum, \cite{ADS}, Brendle-Choi, \cite{BCh}, and Choi-Haslhofer-Hershkovits, \cite{ChHH},  we get uniqueness for ancient flows of surfaces in 
  higher codimension:
  
  \begin{Cor}
  If $M^2_t\subset\RR^N$ is an ancient MCF of surfaces and one tangent flow at $-\infty$ is a cylinder $\SS^1_{\sqrt{2}}\times \RR$, then $M_t\subset \RR^3\subset \RR^N$ for some $3$-plane $\RR^3$.  Therefore, by \cite{ChHH} $M_t$ is either shrinking round cylinders, or the ancient ovals, or the bowl soliton.
 \end{Cor}
  
  White, \cite{W2}, and Haslhofer-Hershkovits, \cite{HH}, constructed ancient MCF of closed hypersurfaces that for time zero disappear in a round point and at time $-\infty$ are shrinking cylinders.  These are  the ancient ovals.    Hershkovits, \cite{H}, showed (see also Haslhofer, \cite{Has}) that the bowl soliton in $\RR^3$ is the unique translating solution  of MCF which has the family of shrinking cylinders as an asymptotic shrinker at $-\infty$.

  \subsection{Rigidity of cylinders}

   Our next result   plays a key role in the proof of Theorem \ref{t:ancientcylinder}
   and in 
  the regularity of MCF in higher codimension, cf. \cite{CM8}.  This result shows that cylinders are rigid in a very strong sense: Any  shrinker, even in a large dimensional space, that is sufficiently close to a cylinder on a large enough, but compact, set is itself   a cylinder.  
  To state the theorem, 
let $\cC_{n,N}$ be the collection of all $\RR^N$ rotations of  $\SS^k_{\sqrt{2k}} \times \RR^{n-k}$ for $k=1, \dots , n$.

  \begin{Thm}	\label{t:rigidity}
  There exists $R_N$ so that if $\Sigma^n \subset \RR^N$ is a complete shrinker with finite entropy and there exists $\cC \in \cC_{n,N}$ so that
  $B_{R_N} \cap \Sigma$ is a graph over $\cC$ of a normal vector field $V$ with $\| V \|_{C^{2,\alpha}} \leq R_N^{-1}$, then $\Sigma \in \cC_{n,N}$.
  \end{Thm}

  The rigidity of cylinders in codimension one was proven in \cite{CIM}.  To prove Theorem \ref{t:rigidity},
   we show that a shrinker, even in high codimension, that is close to a cylinder on a  large bounded set must  be a hypersurface in some Euclidean subspace.   
   
    One of several reasons that cylinders are significant is that they are the most prevalent singularities.  By uniqueness of solutions to ODEs,   any   shrinking curve in $\RR^N$ is planar.   From this and dimension reduction, it is expected that for MCF in all codimension the most prevalent singularities are $\gamma\times \RR^{n-1}$.  Here $\gamma$ is a closed planar curve (Abresch-Langer) that is a round circle if embedded or stable, \cite{CM6}, or with $\lambda (\gamma) < 2$.

 \subsection{Further applications}  Even for hypersurfaces, singularities of MCF are too numerous to classify.     The hope is that the generic ones that cannot be perturbed away are much simpler.   
 Combined, this paper and
 \cite{CM12} give the first general bounds on generic singularities of surfaces in arbitrary codimension.  
    
 \subsubsection{Conjectures}

 Using \cite{CIMW} and Brendle, \cite{B},  Bernstein-Wang, \cite{BW3}, showed that any shrinker in $\RR^3$ with entropy $\leq \lambda (\SS^1)+\epsilon$, is a flat plane, round sphere, or round cylinder.   We believe that there should be a similar classification in low dimension and any codimension of low entropy shrinkers (cf. conjecture $0.10$ in \cite{CIMW}):   
 
 \begin{Con}  \label{c:bigC1}
  There exists $\epsilon>0$ so that for $n\leq 4$ and any codimension, the only shrinkers
  with entropy $< \lambda (\SS^1)+\epsilon$
are  round generalized cylinders, $\SS^{k}_{\sqrt{2\,k}}\times \RR^{n-k}$. 
  \end{Con} 
  
  We conjecture that for any $n$ the round $\SS^n$ has the least entropy of any closed shrinker{\footnote{ In \cite{CIMW}, it was conjectured that the round $\SS^n$ minimizes entropy 
 among closed hypersurfaces  for $n \leq 6$.  This   was proven by 
   Bernstein-Wang, \cite{BW1}.   Zhu later proved this for all $n$ in \cite{Z}; cf. \cite{BW2},  \cite{KZ}.  We conjecture that this holds in all codimension.}}
 $\Sigma^n \subset \RR^N$.  
This was proven for hypersurfaces in \cite{CIMW}; see also \cite{HW}.     The ``Simons cone'' over $\SS^2\times \SS^2$ has entropy $<\lambda (\SS^1)$, see \cite{CIMW}.  So already for $n=5$,  round cylinders do give not a complete list of the lowest entropy shrinkers.
Conjecture \ref{c:bigC1} is known for $n=1$ since shrinking curves  are  planar and have entropy $\geq \lambda (\SS^1)$.

 Conjecture \ref{c:bigC1}  combined with Theorem \ref{t:ancientcylinder} would imply that any ancient solution $M_t^n \subset \RR^N$ with entropy at most $\lambda (\SS^1) + \epsilon$  is   a hypersurface in a Euclidean subspace provided $n\leq 4$.  This would give that all blowups near any cylindrical singularity for $n\leq 4$ are ancient flows of hypersurfaces.    Thus, reducing the system to a single equation.

 Finally, we conjecture:
  
  \begin{Con}	\label{c:bigC}
The optimal constant $C_n$ in Corollary \ref{c:counting2} satisfies
  \begin{align}
  	C_n \, \lambda ( \SS^1) < n+2 \, .
  \end{align}
  \end{Con}
  
  If this conjecture holds, then any ancient solution $M_t^n \subset \RR^N$ with entropy $<\lambda (\SS^1) + \epsilon$, would be a flow of hypersurfaces in a Euclidean subspace.  This would give that all blowups near any cylindrical singularity are ancient flows of hypersurfaces.    

  \section{The operator $\cL$  on the Gaussian space on  shrinkers}
  
   The shrinker equation is $\bH = \frac{ x^{\perp}}{2}$, 
where $\bH = -\Tr \, A$ is the mean curvature vector, $A$ the second fundamental form, and $x^{\perp}$ is the perpendicular part of $x$.{\footnote{See \cite{AHW}, \cite{AS}, \cite{LL} and \cite{Wa} for results on higher codimension MCF.}} 
 Set  $f =  \frac{|x|^2}{4}  $, so the Gaussian weight is $\e^{-f}$.
As in lemma $3.20$ in \cite{CM6}, the coordinate functions $x_i$ are $\frac{1}{2}$-eigenfunctions and   $|x|^2 - 2n$ is a $1$-eigenfunction for $\cL$ on any shrinker with finite entropy.   
We will need some standard facts about $L^2$ eigenfunctions (cf. section $3$ in \cite{CM7}):

\begin{Lem}	\label{l:std}
If $\Sigma^n \subset \RR^N$ is a shrinker   and $u$ is a $L^2$ $\mu$-eigenfunction, then $u \in W^{1,2}$ and $\int |\nabla u|^2 \, \e^{-f} =\mu \, \int u^2 \, \e^{-f}$.
If $v$ is a $\nu$-eigenfunction with $\nu \ne \mu$, then 
\begin{align}	\label{e:lastclaims}
	0 = \int u\,v \, \e^{-f} = \int \langle \nabla u , \nabla v \rangle \, \e^{-f} \, .
\end{align}
\end{Lem}

\begin{proof}
Let $\eta$ be a compactly supported function with $\eta^2 \leq 1$ and $|\nabla \eta | \leq 1$.  Taking the divergence of $u\, \nabla u \, \eta^2 \, \e^{-f}$ and applying Stokes'  theorem gives that
\begin{align}	\label{e:basicU}
	\int |\nabla u|^2 \, \eta^2 \, \e^{-f} - \mu \int u^2\, \eta^2 \, \e^{-f} =- 2 \int u\, \eta\, \langle \nabla u , \nabla \eta \rangle \, \e^{-f} \, .
\end{align}
Applying the absorbing inequality $2\,a\,b \leq \frac{a^2}{2} + 2\,b^2$ and then using $\eta^2 \leq 1$ and $|\nabla \eta| \leq 1$ gives
\begin{align}	 
	\int |\nabla u|^2 \, \eta^2 \, \e^{-f} \leq  \left(  \mu   + 2  \right)  \,  \int u^2\, \e^{-f}  + \frac{1}{2}  \int |\nabla u|^2 \, \eta^2 \, \e^{-f} \, .
\end{align}
Absorbing the last term on the right and   taking   $\eta$'s converging to one everywhere, we conclude that $\int |\nabla u|^2 \, \e^{-f} < \infty$.  Once we have this, 
$|u| \, |\nabla u|$ is also integrable, so the right-hand side of \eqr{e:basicU} goes to zero as   $\eta \to 1$.  We conclude  that 
$\int |\nabla u|^2 \, \e^{-f} = \mu \, \int u^2 \, \e^{-f}$.  Finally, \eqr{e:lastclaims} follows from the symmetry of $\cL$.
\end{proof}

We show next that the $W^{1,2}$ norm   of an eigenfunction  concentrates in a bounded set.  The next lemma and corollary apply  to $W^{1,2}$ functions
that are either entire or defined  on a compact subdomain and vanish on the boundary.

  \begin{Lem}	\label{l:lemmamu}
If $u$ is a $W^{1,2}$ function on a shrinker $\Sigma^n \subset \RR^N$, then 
\begin{align}
	  \int |x|^2 \, u^2 \, \e^{-f}  \leq 4\,n \, \int u^2 \, \e^{-f} + 16 \, \int |\nabla u|^2 \, \e^{-f} \, .
\end{align}
Moreover, for any $r>2$, we have
\begin{align}
	\int_{\Sigma \setminus B_{r}} |\nabla u|^2 \, \e^{-f}  \leq  \int_{\Sigma \setminus B_{r-1}}  (5\, u^2 + (\cL \, u)^2 ) \, \e^{-f}    \, . 
\end{align}
\end{Lem}

\begin{proof}
Since $\cL \,  |x|^2 = 2\,n - |x|^2$, taking the divergence of $ u^2 \, x^T \, \e^{-f}  $ and applying Stokes' theorem and then the absorbing inequality $2\,a\,b \leq \frac{a^2}{4} + 4\,b^2$ gives
\begin{align}
	\frac{1}{2} \, \int |x|^2 \, u^2 \, \e^{-f} &= n \, \int u^2 \, \e^{-f} + 2 \int u \, \langle \nabla u , x^T \rangle \, \e^{-f} \notag \\
	&\leq n \, \int u^2 \, \e^{-f} + \frac{1}{4}  \int u^2 \, |x|^2 \, \e^{-f} + 4\, \int |\nabla u|^2 \, \e^{-f} \, .
\end{align}
The first claim follows.  For the second claim, let $\psi$ be a function that is identically zero on $B_{r-1}$ and identically one outside of $B_r$.  Taking the divergence of 
$ \psi^2 \, u \, \nabla u \, \e^{-f}  $ and 
applying Stokes' theorem and then the absorbing inequalities $ a\,b \leq \frac{a^2}{2} + \frac{b^2}{2}$  and $2\,a\,b \leq \frac{a^2}{2} + 2\,b^2$   gives
\begin{align}
	\int \psi^2 \, |\nabla u|^2 \, \e^{-f} &= -  \int \psi^2\, u\,\cL \, u \, \e^{-f} - 2\, \int   \psi\, \, u\, \langle\, \nabla \psi , \nabla u \rangle \, \e^{-f} \notag \\
		&\leq \frac{1}{2} \, \int \psi^2\, (u^2 + (\cL \, u)^2 ) \, \e^{-f} + \frac{1}{2} \, \int  \psi^2 \, |\nabla u|^2 \, \e^{-f} + 2 \int u^2\, |\nabla \psi |^2 \, \e^{-f} \, . 
\end{align}
Simplifying this and taking $\psi$ to cut off linearly gives the second claim.
\end{proof}

One immediate consequence of Lemma \ref{l:lemmamu} (with $u\equiv 1$)
 is that if $\Sigma^n \subset \RR^N$ is a shrinker with  entropy $\lambda < \infty$, then $\lambda$ is bounded in terms of the volume
 of $B_r \cap \Sigma$ for $r > \sqrt{4n}$.

 \begin{Cor}	\label{c:lemmamu}
If $\cL \, u = - \mu \, u$ on a shrinker $\Sigma^n \subset \RR^N$ and $\|u \|_{L^2}= 1$, then for any $r> 2$  
\begin{align}
	  \int_{\Sigma \setminus B_r} \left\{ u^2 + |\nabla u|^2 \right\} \, \e^{-f} \leq  (6+ \mu^2) \frac{4\,(n+4\,\mu)}{(r-1)^2} \, . 
\end{align}
\end{Cor}

\begin{proof}
Lemma \ref{l:std} gives that $\| \nabla u \|_{L^2}^2 = \mu$. Thus, 
 Lemma \ref{l:lemmamu}
gives that
\begin{align}
	  (r-1)^2 \, \int_{\Sigma \setminus B_{r-1} } u^2 \, \e^{-f} &\leq \int |x|^2 \, u^2 \, \e^{-f}   \leq 4\,n   + 16 \, \mu  \,, \\
	 	\int_{\Sigma \setminus B_{r}} |\nabla u|^2 \, \e^{-f}  &\leq  \int_{\Sigma \setminus B_{r-1}}  (5 + \mu^2) \, u^2   \, \e^{-f}    \, . 
\end{align}
Combining these gives
the corollary.
\end{proof}

    \section{Sharp polynomial growth of eigenfunctions}
    
    On $\RR^n$, the $L^2$ space is spanned by eigenfunctions for $\cL$ and these   are polynomials of degree   twice the eigenvalue.  Moreover, on any shrinker, 
     the coordinate functions are eigenfunctions   with eigenvalue  $\frac{1}{2}$ and $|x|^2-2\,n$ is an eigenfunction with eigenvalue $1$.  In both cases, the degree is twice the eigenvalue.
     The next theorem shows that  $L^2$ eigenfunctions on a shrinker always  grow at most polynomially with degree twice the eigenvalue.

 \begin{Thm}	\label{t:polyGu}
 If $\cL\, u = - \mu \, u$ on a shrinker $\Sigma^n \subset \RR^N$ and $\| u \|_{L^2} < \infty$, then  
 \begin{align}
 	u^2(x) \leq C_n \, \lambda (\Sigma) \,  \| u \|^2_{L^2(\Sigma)} \, (4 + |x|^2)^{2\,\mu} \, .
 \end{align}
 \end{Thm}

 The key to Theorem \ref{t:polyGu} will be to use parabolic estimates on   an associated solution of the heat equation on the self-shrinking MCF.

  \subsection{Separation of variables solutions}
  
  If $u$ is an   eigenfunction on $\RR^n$ with    $\cL\, u = - \mu\, u$, then we get a separation of variables solution $v(x,t)$ of the heat equation
  \begin{align}
  	v(x,t) = (-t)^{\mu} \, u \left( \frac{x}{\sqrt{-t}}\right) \, .
  \end{align}
On $\RR$, 
   $\cL \, x = - \frac{1}{2}\, x$  gives $v(x,t)=x$, while
 $\cL \, ( x^2-2) = - (x^2 -2)$   gives $v(x,t) = x^2 + 2\,t$.
       
    Suppose that $\Sigma^n \subset \RR^N$ is a shrinker and define a MCF of sets  $\Sigma_t = \sqrt{-t} \, \Sigma$.      Let $M_t$ be a MCF associated to $\Sigma_t$.  $M_t$ comes with a parametrization that may not come from scaling.  As sets $M_t$ and $\Sigma_t$ are the same.  
    
    \begin{Lem}	\label{l:sepvar}
    If $u$ is a function on $\Sigma$ with $\cL \, u = - \mu \, u$, then 
  $v$ given on  the   $\Sigma_t$'s by
\begin{align}
	v(y,t) = (-t)^{\mu} \, u \left(\frac{y}{\sqrt{-t}} \right)  
\end{align}
satisfies $(\partial_t - \Delta_{M_t})\, v = 0$ on the MCF $M_t$.
    \end{Lem}

    \begin{proof}
    Given $t< 0$ and a point $y \in \Sigma_t$, we get that
    \begin{align}		\label{e:Hsigt}
    	\bH_{\Sigma_t} (y) = \frac{1}{\sqrt{-t}} \, \bH_{\Sigma}\left( \frac{y}{\sqrt{-t}}\right) = \frac{1}{ \sqrt{-t} } \, \frac{\left[ \frac{y}{ \sqrt{-t} } \right]^{\perp}   }{2} = \frac{y^{\perp}}{-2t} \, .
    \end{align}
    Here, we have freely used that the normal projection $(\cdot)^{\perp}$ operator is invariant under dilation and, thus,  is the same at corresponding points in $\Sigma$ and $\Sigma_t$.
Since $\cL \, u = - \mu \, u$ on $\Sigma$, we have 
\begin{align}	\label{e:DeltaSigmau}
	\Delta_{\Sigma} u (x) = \frac{1}{2} \, \langle \nabla_{\Sigma} u (x) , x^T \rangle - \mu \, u (x) \, .
\end{align}
At $y \in \Sigma_t$, we use the chain rule and \eqr{e:DeltaSigmau} to compute the $\Sigma_t$ Laplacian of $v$  
\begin{align}	\label{e:deltasigt}
	\Delta_{\Sigma_t} v (y,t) &= (-t)^{\mu} \, \Delta_{\Sigma_t} \left[ u \left(\frac{y}{ \sqrt{-t} } \right) \right] 
	=  (-t)^{\mu-1} \, \left[ \Delta_{\Sigma} u \right]   \left(\frac{y}{ \sqrt{-t} } \right)  \notag \\
	&=  (-t)^{\mu-1} \, \left[  \frac{1}{2}   \left\langle \nabla_{\Sigma} u \left(\frac{y}{ \sqrt{-t} } \right) ,  \frac{y^T}{ \sqrt{-t} }  \right\rangle - \mu \, u \left( \frac{y}{ \sqrt{-t} } \right) \right]
	\, .
\end{align}
 If $y(t) \in \Sigma_t$ evolves by   MCF  $y_t = -\bH_{\Sigma_t}(y)$, then \eqr{e:Hsigt} gives
 \begin{align}	\label{e:ytderiv}
 	\partial_t \, \left( \frac{y}{ \sqrt{-t} } \right) = \frac{1}{2} \, (-t)^{ - \frac{3}{2} } \, y + \frac{y_t}{ \sqrt{-t} } = \frac{1}{2} \, (-t)^{ - \frac{3}{2} } \, y + \frac{y^{\perp}}{ 2t\, \sqrt{-t} }  = \frac{1}{ \sqrt{-t} } \, \left( \frac{y^T}{-2t} \right) \, .
 \end{align}
    Therefore, using the chain rule and then \eqr{e:ytderiv}
gives
    \begin{align}		\label{e:comparewith}
    	\partial_t \left[ v (y,t) \right] = \partial_t \left[ (-t)^{\mu} \, u \left(\frac{y}{\sqrt{-t}} \right) \right] &= (-t)^{\mu}  \left\langle \nabla_{\Sigma} u \left(\frac{y}{\sqrt{-t}} \right) , 
	\partial_t   \left( \frac{y}{ \sqrt{-t} } \right) \right\rangle -\mu (-t)^{\mu -1} \, u \left(\frac{y}{\sqrt{-t}} \right)    \notag \\
	&= (-t)^{\mu-1} \, \left\langle \nabla_{\Sigma} u \left(\frac{y}{\sqrt{-t}} \right) , 
	 \frac{y^T}{2\,\sqrt{-t}}  \right\rangle -\mu \, (-t)^{\mu -1} \,  \, u \left(\frac{y}{\sqrt{-t}} \right)
	\, .
    \end{align}
Combining \eqr{e:deltasigt} and \eqr{e:comparewith},  we see that $(\partial_t - \Delta_{\Sigma_t})\,v = 0$  on the MCF.
\end{proof}

 \subsection{Sharp polynomial growth of drift eigenfunctions}

  \begin{Lem}	\label{l:heatconvex}
  If $(\partial_t- \Delta_{M_t})\, w=0$ on a MCF $M_t$ and $q \geq 1$, then $ \left( \partial_t - \Delta_{M_t} \right) \, |w|^q \leq 0$.
  \end{Lem}
  
  \begin{proof}
  Given any function $v: \RR \to \RR$, set $h = v(w^2)$.  
   Differentiating gives
  \begin{align}
  	h_t &= v'(w^2) \, 2 \, w \, w_t  \, , \\
	\nabla_{M_t} h &= v' (w^2) \, 2 \, w \, \nabla_{M_t} w \,  , \\
	\Delta_{M_t}\, h &= v' (w^2)\, (2\,|\nabla_{M_t} w|^2 + 2\, w \, \Delta_{M_t}\, w) + v''(w^2)\, 4\,w^2\, |\nabla_{M_t} w|^2 \, .
  \end{align}
  Therefore, using that $(\partial_t- \Delta_{M_t})\, w=0$, we have
  \begin{align}
  	h_t - \Delta_{M_t}\, h = - 2 \, \left[ v'(w^2)  +  2\, v''(w^2)\,  w^2 \right] \,  |\nabla_{M_t} w|^2 \, .
  \end{align}
  In particular, we have $h_t - \Delta_{M_t}\, h \leq 0$ as long as 
  \begin{align}	\label{e:wantit}
  	v'(s) + 2 \, s \, v''(s) \geq 0 \, .
  \end{align}
 Now, we set $v(s) = s^{ \frac{q}{2} }$ with $q\geq  1$, so that $v'(s)  =  \frac{q}{2} \, s^{ \frac{q-2}{2} }$ and 
	$v''(s) = \frac{q\,(q-2)}{4} \, s^{ \frac{q-4}{2} }$.  
Using this in \eqr{e:wantit}   gives
  \begin{align}
  	  v'(s) + 2 \, s \, v''(s) &=  \frac{q}{2} \, \left[  s^{ \frac{q-2}{2} }  + 2\, s \, \frac{(q-2)}{2} \, s^{ \frac{q-4}{2} }
	  \right]  =  \frac{q}{2} \,  s^{ \frac{q-2}{2} }  \left[  q-1
	  \right]
	  \, .
  \end{align}
  This is nonnegative for $q\geq 1$  as long as it is defined (i.e., $s > 0$ when $q$ is small).  The general case follows by approximation.
  \end{proof}

\begin{proof}[Proof of Theorem \ref{t:polyGu}]
Set $\Sigma_t = \sqrt{-t} \, \Sigma$ and let $M_t$ be the associated MCF.  By Lemma
\ref{l:sepvar}, the function
  $v(y,t) = (-t)^{\mu} \, u\left (\frac{y}{\sqrt{-t}} \right)  $ 
satisfies   $(\partial_t - \Delta_{M_t})\, v = 0$ on   $M_t$.  Thus, Lemma \ref{l:heatconvex} gives that 
$(\partial_t - \Delta_{M_t}) |v| \leq  0$ on   $M_t$, so the weighted monotonicity formula (theorem $4.13$  in \cite{E1}, cf. \cite{Hu}) applies to $|v|$.  Therefore, given any $x_0 \in \Sigma = \Sigma_{-1}$, we get for all $t< -1$ that
\begin{align}
	|u|(x_0) &= |v(x_0 , -1)| \leq (4\,\pi (-1-t))^{ - \frac{n}{2 } } \int_{\Sigma_{t}} |v(y,t)| \, \e^{ \frac{|y-x_0|^2}{4\,(t+1)} } \notag \\
	&= (-t)^{\mu} \, (4\,\pi (-1-t))^{ - \frac{n}{2 } } \int_{\Sigma_{t}} \left|u\left(\frac{y}{\sqrt{-t}}\right)\right| \, \e^{ \frac{|y-x_0|^2}{4\,(t+1)} } \, .
\end{align}
Making the change of variables $x= \frac{y}{\sqrt{-t}}$, we get
\begin{align}
	|u|(x_0) & =   \frac{(-t)^{ \mu + \frac{n}{2}} }{(4\pi (-1-t))^{  \frac{n}{2 } }} \int_{\Sigma} |u(x)| \, \e^{ \frac{|\sqrt{-t}  x-x_0|^2}{4(t+1)} }
	=   \frac{(-t)^{ \mu} }{(4\,\pi\, (1+t^{-1}))^{  \frac{n}{2 } }} \int_{\Sigma} |u(x)| \, \e^{ -\frac{\left| x-\frac{x_0}{\sqrt{-t} } \right|^2}{4\,(1+t^{-1})} } \, .
\end{align}
We will take $t<-4$.  Using this and expanding the square
  gives  
\begin{align}	\label{e:backinherejack}
	|u|(x_0) & \leq    
	  (-t)^{ \mu}  \int_{\Sigma} |u|\,  \e^{ \frac{\langle x , \frac{x_0}{\sqrt{-t} } \rangle }{2\,(1+t^{-1})} } \, \e^{ -\frac{ | x|^2    }{4(1+t^{-1})} } \, .
\end{align}
We will apply the Cauchy-Schwarz inequality to the last term, writing the integrand as the product of $|u| \, \e^{ - \frac{|x|^2}{8}} $ and
$ \e^{ \frac{\langle x , \frac{x_0}{\sqrt{-t} } \rangle }{2\,(1+t^{-1})} } \, \e^{ -\frac{ (1-t^{-1})| x|^2    }{8\,(1+t^{-1})} }$.  The first term just gives $\| u \|_{L^2}$, as desired.
To bound the second term, we use the absorbing inequality 
\begin{align}
	 \left| \left\langle x , \frac{x_0}{\sqrt{-t} } \right\rangle \right| \leq \frac{|x|^2}{8} + 2\, \frac{|x_0|^2}{-t}  \, .
\end{align}
By section $1$ in \cite{CM6} we can bound $\Vol (B_s \cap M_t)$ in terms of a dimensional constant times the entropy times $s^n$.  Combining this all gives
\begin{align}
	\int_{\Sigma} \e^{ \frac{\langle x , \frac{x_0}{\sqrt{-t} } \rangle }{(1+t^{-1})} } \, \e^{ -\frac{ (1-t^{-1})| x|^2    }{4\,(1+t^{-1})} } \leq
	 \e^{  2\, \frac{|x_0|^2}{ -t-1 } } \, \int_{\Sigma} \e^{ -\frac{ (1-2\,t^{-1})| x|^2    }{8\,(1+t^{-1})} } \leq 
	  \e^{   \frac{2\,|x_0|^2}{ -t-1 } } \, \int_{\Sigma} \e^{ -\frac{ | x|^2    }{8 } } \leq C_n \, \lambda (\Sigma) \,  \e^{   \frac{2\,|x_0|^2}{ -t-1 } } 
	\, .
\end{align}
Finally, using this back in \eqr{e:backinherejack}   and taking $t=-4-|x_0|^2$
\begin{align}	 
	u^2(x_0) & \leq    
	  (-t)^{ 2\,\mu} \, \| u \|_{L^2}^2 \, C_n \, \lambda (\Sigma) \,  \e^{   \frac{2\,|x_0|^2}{ -t-1 } } \leq (|x_0|^2 + 4)^{2\,\mu} \, 
	   \| u \|_{L^2}^2 \, C_n \, \lambda (\Sigma) \, \e^2 
 \, .
\end{align}
\end{proof}

\begin{Cor}	\label{c:polyGu}
 If $\cL\, u = - \mu \, u$ on a shrinker $\Sigma^n \subset \RR^N$ and $\| u \|_{L^2} < \infty$, then 
  $v$ given on   $\Sigma_t=\sqrt{-t}\,\Sigma$'s by $v(y,t) = (-t)^{\mu} \, u \left(\frac{y}{\sqrt{-t}} \right)$
is in $\cP_{2\,\mu}$ and satisfies
 \begin{align}
 	v^2(y,t) \leq C_n \, \lambda (\Sigma) \,  \| u \|^2_{L^2(\Sigma)} \, (-4\,t + |y|^{2})^{2\,\mu} \, .
 \end{align}
 \end{Cor}
 
 \begin{proof}
 This follows by combining Theorem \ref{t:polyGu} and Lemma \ref{l:sepvar}.  
 \end{proof}

   \section{Growth and Gaussian inner products}
    
     In  this section, $M_t^n \subset \RR^N$ is an ancient MCF with finite entropy and  $\phi = \bH + \frac{x^{\perp}}{2t}$.  We will study the growth of 
     caloric functions on $M_t$ in the Gaussian $L^2$ norm.  The key result, inspired by \cite{CM2}, uses  linear independence and polynomial growth to produce orthonormal caloric functions with a fixed doubling property.  This will be used in the next section to bound $\dim \cP_d$  and then 
     used   later for sharp bounds on $\cP_1$.

     We will need  
the weighted Huisken monotonicity formula (theorem $4.13$ in \cite{E1}, cf. \cite{Hu}) 
 \begin{align} \label{e:eckerH}
  \frac{d}{dt} \, \left\{   \left( -4\pi \, t \right)^{ - \frac{n}{2} } \, \int_{M_t} v \, \e^{ \frac{|x|^2}{4t} }  \right\} = \left( -4\pi \, t \right)^{ - \frac{n}{2} } \, \int_{M_t} \left\{ (v_t - \Delta\, v) - v \, \left| \phi \right|^2      \right\}\, \e^{ \frac{|x|^2}{4t} }   \, .
 \end{align}

   \subsection{Polynomial growth}
   
   Given $u$, $v \in    L^2(M_t)$, define a bilinear form $J$ and associated quadratic form $I$  by 
 \begin{align}
         J_t (u , v) &=  \left( -4\pi \, t \right)^{ - \frac{n}{2} } \, \int_{M_t} u \, v \, \e^{ \frac{|x|^2}{4t} }\, , \\
  I_u (t) &= \left( -4\pi \, t \right)^{ - \frac{n}{2} } \, \int_{M_t} u^2 \, \e^{ \frac{|x|^2}{4t} } \, .
 \end{align}

The next lemma shows that $I_u$ is monotone and grows polynomially when $u \in \cP_d$.
 
 \begin{Lem}	\label{l:iprimelem}
 If  $u \in \cP_d$, then there exists $C_{u,n,d}$  so that
 \begin{align}
 I_u (t)&\leq  C_{u,n,d} \, \lambda (M_t) \, (1-t)^d \, ,\\
  I_u' (t) &= \left( -4\pi \, t \right)^{ - \frac{n}{2} } \, \int_{M_t} \left\{ -2 \, |\nabla u|^2 - u^2 \, |\phi |^2 \right\}  \, \e^{ \frac{|x|^2}{4t} } \leq 0  \, . \label{e:Ipri}
\end{align}
\end{Lem}

\begin{proof}
Since $u \in \cP_d$, there is some $C_u$  so that $ |u(x,t)| \leq C_u \, (1 + |x|^d+|t|^{\frac{d}{2}})$ and, thus,
\begin{align} \label{e:Iupolyd}
I_u (t) &\leq C_u \, (-t)^d \,  \left( -4\pi \, t \right)^{ - \frac{n}{2} } \, \int_{M_t}   \e^{ \frac{|x|^2}{4t} } + C_u \, \left( -4\pi \, t \right)^{ - \frac{n}{2} } \, \int_{M_t} (1+|x|^d)^2 \, \e^{ \frac{|x|^2}{4t} } \notag \\
&\leq C_u \, \lambda (M_t)  \, (-t)^d + C_u \, \left(  4\pi \right)^{ - \frac{n}{2} } \, \int_{ \frac{M_t}{\sqrt{-t}}} (1+ (-t)^{\frac{d}{2}} \, |y|^d)^2 \, \e^{ - \frac{|y|^2}{4} } \leq  C_{u,n,d} \, \lambda (M_t) \, (1-t)^d \, ,
\end{align}
where $C_{u,n,d}$ depends on $u$, $n$ and $d$ but not on $t$.  
 Applying \eqr{e:eckerH} with $v=u^2$ gives \eqr{e:Ipri}.
\end{proof}

 \subsection{General constructions}
 
  Let $u_0 \equiv 1, \, u_1 , \dots , u_{\ell} \in \cP_d (M_t)$ be linearly independent.  These are independent, but not necessarily orthogonal.  To separate them,  we will  use ideas introduced in section $4$ in \cite{CM2} for studying   harmonic functions.
  
  Following definition $4.2$ in \cite{CM2}, for each $t_0$ we set $w_{0,t_0} = u_0 = 1$ and then inductively define $w_{i,t_0}$ by
  choosing coefficients  $\lambda_{j,i}(t_0) \in \RR$ so that 
  \begin{align}
  	 w_{i,t_0} \equiv u_i  - \sum_{j=0}^{i-1} \lambda_{j,i}(t_0) \, u_j 
\end{align}
 is $J_{t_0}$-orthogonal to $u_0 , \dots , u_{i-1}$.
Finally, 
 set $f_i (t_0) = I_{w_{i,t_0}}(t_0)$.  
 
 Following proposition $4.7$ in \cite{CM2}, we get the following properties:
 \begin{enumerate}
 \item If $t_0 \leq t_1$, then 
$
 	f_i (t_1) =  I_{w_{i,t_1}} (t_1) \leq I_{w_{i,t_0}} (t_1) \leq f_i (t_0)$.
	  \item For each $i$, there exist $T_i$ and $C_i$ so that  for $t \leq T_i$ we have $0 < f_i (t) \leq C_i \, (1 -t)^d$.
 \end{enumerate}
 
 \vskip2mm
 The next lemma is a variation on proposition $4.16$ in \cite{CM2}   adapted to our situation:
 
 \begin{Lem}	\label{l:goodm}
 Given $\delta > 0$ and $\Omega > 1$, there exist $m_{q} \to \infty$ so that $v_1 , \dots , v_{\ell}$ defined by $v_i = \frac{ w_{i,-\Omega^{m_q +1}}}{ \sqrt{f_i (-\Omega^{m_q +1})}}$ satisfy
  \begin{align}	\label{e:goodm}
   J_{-\Omega^{m_q+1}} (v_i , v_j) = \delta_{ij}  
   {\text{ and 
}}
 	\sum_{i=1}^{\ell} I_{v_i} (-\Omega^{m_q}) \geq \ell \, \Omega^{-d-\delta} \, .
 \end{align}
 \end{Lem}

  \begin{proof}
  By (1) and (2), the sequence $a_m = \Pi_{i=1}^{\ell} f_i (-\Omega^m)$ is non-decreasing.  By (2) positive for $m$ large, and   $a_m \leq C \, (1+\Omega^m)^{d\, \ell}$.  Therefore, there must exist $m_{q} \to \infty$ where
  \begin{align}	\label{e:polygrow}
  	a_{m_q +1} \leq  \Omega^{d\, \ell + \delta} \, a_{m_q}  \, .
  \end{align}
  If this was not the case, then we would get some $\bar{m}$ so that $a_{m+1} \geq \Omega^{d\, \ell+ \delta} \, a_m$ for every $m \geq \bar{m}$.  Iterating this forces $a_m$ to grow and, eventually, contradict
  $a_m \leq C \, (1+\Omega^m)^{d\, \ell}$.
  
  We will show    \eqr{e:goodm} holds  for $m_q$ satistying \eqr{e:polygrow}.  Namely, (1) and \eqr{e:polygrow} give
  \begin{align}
  	\prod_{i=1}^{\ell}  \, I_{v_i}(-\Omega^{m_q}) \geq \prod_{i=1}^{\ell} \, \frac{ f_i (-\Omega^{m_q})}{ f_i (-\Omega^{m_q+1})} = \frac{a_{m_q}}{a_{m_q+1}} \geq \Omega^{-d\, \ell - \delta} \, .
  \end{align}
  Finally, combining this with the arithmetic-geometric mean inequality gives
  \begin{align}
  	\frac{1}{\ell} \, \sum_{i=1}^{\ell} I_{v_i} (-\Omega^{m_q}) \geq \left( \prod_{i=1}^{\ell}  \, I_{v_i}(-\Omega^{m_q} ) \right)^{ \frac{1}{\ell} } \geq \Omega^{-d - \frac{ \delta}{\ell} }  \, .
  \end{align}
  \end{proof}

       \subsection{Localization}
 
We will need the following localization inequality:
 
 \begin{Lem}	\label{l:localizeMt}
 Given any function $u$ on $M_t$, we have
 \begin{align}
 	 \frac{1}{-t} \int_{M_t} |x|^2 \, u^2 \,  \e^{   \frac{|x|^2}{4t} } \leq 4n \, \int_{M_t}   u^2 \,  \e^{   \frac{|x|^2}{4t} } - 4t \, \int_{M_t}  \left( 4 \, |\nabla u|^2 + u^2 \, |\phi|^2 \right)  \e^{   \frac{|x|^2}{4t} } \, .
 \end{align}
 \end{Lem}
 
 \begin{proof}
 Using that $(\partial_t - \Delta_{M_t}) |x|^2 = - 2n$ and $x_t = - \bH$ on $M_t$, we get
 \begin{align}	\label{e:dvMt}
 	2\, \e^{ -  \frac{|x|^2}{4t} } \,  \dv_{M_t} \, \left( u^2 \, x^T \, \e^{   \frac{|x|^2}{4t} } \right) &= 4 \, u \, \langle \nabla u , x^T \rangle + u^2 \, \left( \Delta_{M_t} |x|^2 + \frac{|x^T|^2}{t} \right) \notag \\
	&= 4 \, u \, \langle \nabla u , x^T \rangle + u^2 \, \left( 2n - 2 \langle x^{\perp} , \bH \rangle + \frac{|x^T|^2}{t} \right)  \\
	&= 4 \, u \, \langle \nabla u , x^T \rangle + u^2 \, \left( 2n - 2 \langle x^{\perp} , \phi \rangle + \frac{|x|^2}{t} \right) 
	\, . \notag
 \end{align}
 Using the absorbing inequality twice gives 
 \begin{align}
 	\left| 4 \, u \, \langle \nabla u , x^T \rangle - 2\, u^2  \langle x^{\perp} , \phi \rangle  \right|  &\leq \frac{u^2 \, |x^T|^2}{2|t|} - 8t \, |\nabla u|^2 + \frac{u^2\,  |x^{\perp}|^2}{2|t|} - 2 \, t \, u^2 \, |\phi|^2  \notag \\
	&=   \frac{u^2\, |x|^2}{2|t|} - 8t \, |\nabla u|^2   - 2 \, t \, u^2 \, |\phi|^2 \, .
 \end{align}
Inserting this in \eqr{e:dvMt} and applying the divergence theorem gives the lemma.
  \end{proof}

 \section{Sharp bounds for $\dim \cP_d$ on an ancient MCF}
 
 In this section,  $M_t^n \subset \RR^N$ is an ancient MCF with $\lambda (M_t) \leq \lambda_0$ for all $t$.
   We will  need the following local meanvalue inequality (proposition $2.1$ in \cite{E2}; cf. \cite{Hu}):

\begin{Lem}	\label{l:mvi}
There exists $c$ depending on $n$ so that if $(\partial_t - \Delta)\, u = 0$, then for any $\rho > 0$
\begin{align}
	u^2 (x_0 , t_0) \leq \frac{c}{\rho^{n+2}} \, \int_{t_0 -\rho^2}^{t_0} \int_{B_{\rho}(x_0) \cap M_t} u^2 \, .
\end{align}
\end{Lem}

 \begin{proof}[Proof of Theorem \ref{t:maria}]
 Suppose that $d\geq 1$ and $u_0 \equiv 1 , u_1 , \dots , u_{\ell}$ are linearly independent functions in $\cP_d (M_t)$.  We will prove that there is a constant $C_n$ so that
$
 	\ell \leq C_n \, \lambda_0 \, d^n$.
	
	  The first step is to apply Lemma \ref{l:goodm} with 
 $\Omega = 1 + \frac{3}{d}$ and $\delta = d$ to get $m_{q} \to \infty$ so that $v_1 , \dots , v_{\ell}$ defined by $v_i = \frac{ w_{i,-\Omega^{m_q +1}}}{ \sqrt{f_i (-\Omega^{m_q +1})}}$ satisfy
  \begin{align}	\label{e:goodm2}
   J_{-\Omega^{m_q+1}} (v_i , v_j) = \delta_{ij}  
   {\text{ and 
}}
 	\sum_{i=1}^{\ell} I_{v_i} (-\Omega^{m_q}) \geq \ell \, \Omega^{-2d} \geq \e^{-6} \, \ell  \, .
 \end{align}
Integrating $I_{v_i}'$ from $-(1+ 1/d)\Omega^{m_q}$ to $-\Omega^{m_q}$, there exists $t_0 \in [-(1+ 1/d)\Omega^{m_q} , -\Omega^{m_q}]$ with
\begin{align}	\label{e:getlocali}
	\frac{\Omega^{m_q}}{d}  \,\sum_{i=1}^{\ell} \left( -4\pi \,t_0 \right)^{ - \frac{n}{2} } \, \int_{M_{t_0}} \left\{  2 \, |\nabla v_i|^2 + v_i^2 \, |\phi |^2 \right\}  \, \e^{ \frac{|x|^2}{4t_0} }  &=  \frac{\Omega^{m_q}}{d} \,  \, \sum_{i=1}^{\ell} \left| I'_{v_i}\right| (t_0) \notag \\
	&\leq \int_{-(1+ 1/d)\Omega^{m_q}}^{-\Omega^{m_q}} \sum_{i=1}^{\ell} \left| I'_{v_i} \right| (t)\, dt \leq \ell \, .
\end{align}
Since $I_{v_i}$ is monotone and $\frac{|t_0|}{\Omega^{m_q}} \in [1,1+1/d]$, \eqr{e:goodm2} and \eqr{e:getlocali} give
  \begin{align}		\label{e:arrange1}
   \e^{-6} \, \ell  &\leq  \sum_{i=1}^{\ell} I_{v_i} (t_0)   \, , \\
   -t_0 \,  \left( -4\pi \, t_0 \right)^{ - \frac{n}{2} } \,   \sum_{i=1}^{\ell} \int_{M_{t_0}} \left\{  2 \, |\nabla v_i|^2 + v_i^2 \, |\phi |^2 \right\}  \, \e^{  \frac{|x|^2}{4t_0}}   & \leq d\, \ell \left( 1 + \frac{1}{d} \right) \leq 2 \, d \, \ell \, .	\label{e:arrange2}
   \end{align}
  Applying the localization inequality Lemma \ref{l:localizeMt} to each $v_i$  gives
 \begin{align}	 
 	    \left( -4\pi \, t_0 \right)^{ - \frac{n}{2} } \,  \int_{M_{t_0}} \frac{|x|^2}{-t_0} \, v_i^2 \,  \e^{  \frac{|x|^2}{4t_0}} &\leq 4n \,   I_{v_i}(t_0)-t_0\,   \left( -4\pi \,t_0 \right)^{ - \frac{n}{2} } \,   \int_{M_{t_0}}  \left( 16 \, |\nabla v_i|^2 + 4\, v_i^2 \, |\phi|^2 \right)  \e^{  \frac{|x|^2}{4t_0}} \notag \, .
 \end{align}
 Summing this over $i$ and   then using \eqr{e:arrange2}, we conclude that
  \begin{align}	 	\label{e:localix}
 	    \left( -4\pi \, t_0\right)^{ - \frac{n}{2} } \,  \int_{M_{t_0}} \frac{|x|^2}{-t_0} \,   \sum_{i=1}^{\ell} v_i^2 \,  \e^{  \frac{|x|^2}{4t_0}}
	    &\leq (4n   +  16\,d )\, \ell  \leq (4n   +  16 ) \, d\, \ell  \, .
 \end{align}
Now, define the function $K(x,t)=    \sum_{i=1}^{\ell} v_i^2 (x,t)$ to be the ``trace of the Bergman kernel''.   Equation \eqr {e:arrange1} gives 
\begin{align}	\label{e:Kfrombelow}
	  \e^{-6} \, \ell \leq  \left( -4\pi \,t_0 \right)^{ - \frac{n}{2} } \,  \int_{ M_{t_0}  }  K \,  \e^{  \frac{|x|^2}{4t_0} } \, .
\end{align}
 To bound $\ell$, we will combine \eqr{e:Kfrombelow} with an upper bound on the integral of $K$.  We will divide the integral into an inner ball of radius proportional to $\sqrt{-d\,t_0}$ and an integral outside.
 
Set $\Lambda = \e^6 \, (8n+32)$.  It follows from \eqr{e:localix} that
\begin{align}	\label{e:localix2}
	 \left( -4\pi \,t_0 \right)^{ - \frac{n}{2} } \,  \int_{ M_{t_0} \setminus B_{\sqrt{- \Lambda \, d \, t_0}}}  K \, \e^{  \frac{|x|^2}{4t_0}} &\leq   \frac{(4n   +  16 ) \, d}{ \Lambda \, d \, } \, \ell   \leq \frac{\e^{-6}}{2} \, \ell \, .
\end{align}
Suppose, on the other hand, that $x_0 \in B_{\sqrt{ -\Lambda \, d\, t_0}}$.  Since $K(x_0,t_0)$ is the trace of a quadratic form, there exist coefficients $a_1 , \dots , a_{\ell}$  so that
$\sum a_i^2 = 1$ and 
  $u(x,t)= \sum_{i=1}^{\ell} a_i \, v_i (x,t)$ satisfies $K(x_0,t_0) = u^2(x_0,t_0)$.  Moreover,   monotonicity of $I_u$  and \eqr{e:goodm2} give
\begin{align}	\label{e:theboundonIu}
	\sup_{t \geq (1+1/d)t_0} \, I_u (t) \leq I_u ((1+1/d)t_0) \leq I_u (-\Omega^{m_q+1}) = 1 \, .
\end{align}
Set $\rho = \frac{\sqrt{-t_0}}{\sqrt{d}}$ and  observe that there is a constant $c_n$,  depending just on $n$, so that
\begin{align}	\label{e:sillycn}
	\left(  -4\pi \,t_0\right)^{ - \frac{n}{2} } \,  \e^{  \frac{|x_0|^2}{4\,t_0} } \leq c_n \, \sup_{  B_{\rho}(x_0) \times [ t_0 - \rho^2 , t_0] } \, \left\{ \left( - 4\pi \, t \right)^{ - \frac{n}{2} } \,  \e^{ \frac{|x|^2}{4t}}  \right\} \, .
\end{align}
Lemma \ref{l:mvi}
gives $c$ depending on $n$ so that 
\begin{align}
	u^2 (x_0 , t_0) \leq \frac{c}{\rho^{n+2}} \, \int_{t_0 -\rho^2}^{t_0} \int_{B_{\rho}(x_0) \cap M_t} u^2
		  \, .
\end{align}
Combining this with \eqr{e:sillycn} and the bound \eqr{e:theboundonIu} on $I_u$ gives
\begin{align}
	\left(  -4\pi \, t_0\right)^{ - \frac{n}{2} } \,  \e^{  \frac{|x_0|^2}{4t_0} } \, u^2 (x_0 , t_0) &\leq \frac{c\, c_n}{\rho^{n+2}} \, \int_{t_0 -\rho^2}^{t_0}  \left( - 4\pi \, t \right)^{ - \frac{n}{2} } \int_{B_{\rho}(x_0) \cap M_t} u^2\, \e^{ \frac{|x|^2}{4t}}  \leq \frac{c\, c_n}{\rho^{n+2}} \, \int_{t_0 -\rho^2}^{t_0}I_u(t)  \notag \\
	&\leq  \frac{c\, c_n}{\rho^{n}} \,  I_u(t_0 -\rho^2) \leq  \frac{c\, c_n}{\rho^{n}} = c\, c_n \, \left( \frac{d}{-t_0} \right)^{ \frac{n}{2}}
		  \, .
\end{align}
Integrating this bound over $x_0 \in B_{\sqrt{ -\Lambda \, d \, t_0}} \cap M_{t_0}$ gives
\begin{align}	\label{e:Kinside}
	\left( -4\pi\,t_0 \right)^{ - \frac{n}{2} } \, \int_{ B_{\sqrt{ -\Lambda \, d \, t_0} \cap M_{t_0}} } K \, \e^{  \frac{|x|^2}{4t_0}} \leq c\, c_n \, \left( \frac{d}{-t_0} \right)^{ \frac{n}{2}}
	 \, \Vol \, \left( B_{\sqrt{ -\Lambda \, d \, t_0} \cap M_{t_0}}  \right) \leq C_n \, \lambda_0 \, d^n \, .
\end{align}
Using the lower bound from \eqr{e:Kfrombelow} and 
combining \eqr{e:localix2} with \eqr{e:Kinside}, we see that
\begin{align}	 
	 \e^{-6} \, \ell \leq \left( -4\pi \,t_0 \right)^{ - \frac{n}{2} } \,  \int_{ M_{t_0}  }  K \,  \e^{  \frac{|x|^2}{4t_0}}&\leq    \frac{\e^{-6}}{2} \, \ell + C_n \, \lambda_0 \, d^n\, .
\end{align}
We can absorb the first term on the right and the theorem follows.
 \end{proof}

   \section{Entropy controls  spectral multiplicity and heat kernel}

We will next bound the counting function on a shrinker and then estimate the heat kernel.

   \begin{proof}[Proof of Theorem \ref{t:countingN}]
   The shrinker $\Sigma$ gives rise to a MCF $M_t$ where each $M_t$ is given as a set by $\sqrt{-t} \, \Sigma$.  Fix some $\mu \geq 1$.  For each $L^2$-eigenvalue $\mu_i\leq \mu$ of $\cL$ on $\Sigma$, let $u_i$ be an eigenfunction with $\| u_i \|_{L^2(\Sigma)} = 1$.  
   Corollary \ref{c:polyGu}
then gives 
  $w_i \in \cP_{2\mu_i} (M_t)$ defined by  $w_i(y,t) = (-t)^{\mu_i} \, u_i \left(\frac{y}{\sqrt{-t}} \right)$.  Combining this with Theorem \ref{t:maria} gives $C=C(n)$ so that
  \begin{align}
  	\cN ( \mu) \leq \dim \cP_{2\mu}(M_t) \leq C \, \lambda (\Sigma) \, \mu^{n} \, .
  \end{align}
   \end{proof}

 From Theorem \ref{t:countingN}, and the   proof of the Courant nodal domain theorem, \cite{CtHi}, we get:
 
 \begin{Cor}
 If $\Sigma^n\subset \RR^N$ is a shrinker, then any hyperplane through the origin cannot divide $\Sigma$ into more than $C_n\,\lambda(\Sigma)$ many components.  
 \end{Cor}
 
 \begin{proof}
 After a rotation, we may assume that the hyperplane is $\{x_1=0\}$.  Since $\cL\,x_1=-\frac{1}{2}\,x_1$ and $\cN\left(\frac{1}{2}\right)\leq C_n \,\lambda (\Sigma)$ by Theorem \ref{t:countingN}, the claim follows from the argument in the Courant nodal domain theorem for the operator $\cL$; see page 45 of \cite{Cg} for a proof for $\Delta$.  
 \end{proof}
 
 In the case where $\Sigma^n\subset B_{\sqrt{2\,n}}\subset \RR^N$ is a closed minimal submanifold and the entropy reduces to the volume, this result was established by Cheng-Li-Yau in corollary 6 of \cite{CgLYa}.

 \subsection{Drift heat kernel on shrinkers}
 
 In \cite{CgLYa}, Cheng, Li and Yau proved heat kernel estimates on closed spherical minimal submanifolds.   Although there are many of these submanifolds, they form a relatively small subset of all closed shrinkers.   In addition, there are many  non-compact shrinkers.  On a closed manifold, the heat kernel  is given by
 $H(x,y,t) = \sum_i \e^{-\mu_i \, t} \, u_i (x) \, u_i(y)$ where the $u_i$'s are eigenfunctions with eigenvalues $\mu_i$.  In general, the heat kernel on a non-compact manifold cannot be constructed this way unless   the eigenvalues go to infinity at a rate.   The heat kernel  has four properties: 
 $H_t = \cL H$,  $H(x,y,t)=H(y,x,t)$, the reproducing property as $t\to 0$, and the semi-group property.
 
We next   estimate  the drift heat kernel $H $ on a shrinker  in arbitrary codimension.  To construct $H $, we  need that the spectrum is discrete; this was proven by Cheng-Zhou, \cite{CxZh}.
 
 \begin{Thm}	\label{t:heatkernel}
 Let $\Sigma^n \subset \RR^N$ be a shrinker with finite entropy.
There is a complete basis of $W^{1,2}$ eigenfunctions $u_i$ for $\cL$ with eigenvalues $\mu_i$  and $\| u_i \|_{L^2} = 1$.  
The heat kernel $H(x,y,t)$ for $\partial_t - \cL$ exists and is given by
\begin{align}	\label{e:hxyt}
	H(x,y,t)=\sum_{i}\e^{-\mu_i\,t}\, u_i(x)\,u_i(y)\,  .
\end{align}
\end{Thm}

\begin{proof}
Let $\mu_i^j$ be the Dirichlet eigenfunctions for $\cL$ on $B_j \cap \Sigma$ and let $u_i^j$ be the corresponding eigenfunctions with $\| u_i^j \|_{L^2} = 1$.
By domain monotonicity of eigenvalues,  $\mu_i^j$ is non-increasing in $j$ and we get limits $\mu_i = \lim_{j\to \infty} \mu^j_i$.
  For each $i$, elliptic theory gives uniform estimates for the $u_i^j$ on compact subsets and, thus,
Arzela-Ascoli gives   limiting functions $u_i$ with $\cL \, u_i = - \mu_i \, u_i$ with  $\| u_i \|_{L^2} \leq 1$.   Corollary \ref{c:lemmamu} gives that
$\| u_i \|_{L^2} = 1$ and $\| \nabla u_i \|_{L^2} \ne 0$ for $i>0$, as desired. The $\mu_i$ must go to infinity by Theorem \ref{t:countingN}.

We will show   that the $u_i$'s are complete.  If this was not the case, then there would be some with $\| w \|_{L^2} = 1$, $\| \nabla w \|_{L^2} < \infty$, and 
\begin{align}	\label{e:www}
	\int_{\Sigma} u_i \, w \, \e^{-f} = 0 {\text{ for every }} i .
\end{align}
Since  $\mu_i \to \infty$, we can fix  $k$ so that $\mu_k > 2 \, \| \nabla w \|_{L^2}^2$.
The first claim in Lemma \ref{l:lemmamu} gives  
\begin{align}	\label{e:localizew}
	  \int |x|^2 \, w^2 \, \e^{-f}  \leq 4\,n  + 16 \, \int |\nabla w|^2 \, \e^{-f} \, .
\end{align}
Let $\phi_j$ be a  cutoff function that is one on $B_{j-1}$ and zero $\partial B_j$ and set $w_j = \phi_j \, w$. 
It follows from \eqr{e:www}, \eqr{e:localizew} and the uniform convergence on compact sets of 
$u_i^j$'s to $u_i$ that
\begin{align}	\label{e:www2}
	\lim_{j\to \infty}  \, \|   w_j \|_{L^2}  &= 1 \, , \\
	\limsup_{j\to \infty}  \, \| \nabla w_j \|^2_{L^2}  &\leq\| \nabla w \|_{L^2}^2 \, , \\
	\lim_{j\to \infty} \, \int_{\Sigma} u_i^j w_j \, \e^{-f}  &= 0 {\text{ for   }} i \leq k.
\end{align}
In particular, we can choose some $j$ large so that the orthogonal projection $\bar{w}_j$ of $w_j$ onto the eigenspaces with $\mu_i^j$ with $i>k$ has 
\begin{align}
	\frac{3}{4} &<  \| \bar{w}_j \|_{L^2}^2  {\text{ and }}
	 \| \nabla \bar{w}_j  \|_{L^2}^2    \leq \frac{5}{4} \| \nabla w \|_{L^2}^2  \, . \label{e:cont2}  
\end{align}
However, since $\mu_k^j > \mu_k > 2 \, \| \nabla w \|_{L^2}^2$, the variational characterization of eigenvalues gives 
\begin{align}
	2 \, \| \nabla w \|_{L^2}^2 \,  \| \bar{w}_j \|_{L^2}^2 \leq  \| \nabla \bar{w}_j \|_{L^2}^2 \, .
\end{align}
This contradicts  \eqr{e:cont2}, so we conclude that the $u_i$'s are complete.

To see that the sum \eqr{e:hxyt} converges for each $t> 0$, observe that elliptic theory and the bounds $\int_{B_R} u_i^2 \leq \e^{ \frac{R^2}{4} }$ and 
$\int_{B_R} |\nabla u_i|^2 \leq \mu_i \, \e^{ \frac{R^2}{4} }$ give    $c=c(R)$ so that
\begin{align}
	\mu_i \, \sup_{B_R} |u_i|^2 +
	\sup_{B_R} |\nabla u_i|^2 &\leq c \, \mu_i^{ \frac{n}{2} +1} \, .
\end{align}
  Theorem \ref{t:countingN} gives $C=C(n)$ so that $\cN (m) \leq C \, \lambda \, m^n$, so we have
\begin{align}
	\sup_{B_R \times B_R} |H| (x,y,t) &\leq \sum_{m} \left\{ \sum_{\mu_i \in [m-1,m]} c \, \mu_i^{ \frac{n}{2} }\e^{-\mu_i \, t} \right\} \leq 
		 c\, \sum_{m}  \cN(m)  \,  m^{ \frac{n}{2} }\e^{ -(m-1) \, t}   \notag \\
		 &\leq c \, C \, \lambda \, \e^t \, \sum_{m}   m^{ \frac{3n}{2} }\e^{-mt}   \, .
\end{align}
This is finite for each $t>0$.  Arguing similarly gives  estimates also for higher derivatives, so Arzela-Ascoli gives convergence of \eqr{e:hxyt} on compact subsets.  It then follows  that $H$ has the semi-group property and satisfies the drift heat equation.  Finally, the reproducing property at $t=0$ follows from the the completeness of the eigenvalues.
\end{proof}

  \section{Rigidity}

  In this section, we show that a shrinker, even in high codimension, that is close to a cylinder on a sufficiently large bounded set must  be a hypersurface in some Euclidean subspace.  
  It then follows from  \cite{CIM}    that it must  be a cylinder; cf. \cite{CM7}, \cite{GKS}.

  \subsection{Convergence of the spectrum}
  
  By Theorem \ref{t:heatkernel}, the operator $\cL$ on a shrinker has eigenvalues $0=\mu_0 < \mu_1 \leq \dots$ going to infinity and a complete basis of $L^2$ eigenfunctions.
Given $r> \sqrt{2n}$, let $\beta^r_i$ be the Dirichlet eigenvalues of $\cL$ on $B_r \cap \Sigma$.
We show next that the Dirichlet spectrum converges  uniformly.

\begin{Lem}	\label{p:specD}
Given $k$, $\delta > 0$ and $n$, there exists $\bar{r} = \bar{r}(\mu_k , n , \delta)$ so that for $\bar{r} \leq r$
\begin{align}	\label{e:thedist}
	\mu_i \leq \beta_i^r \leq \mu_i + \delta {\text{ for every }} i \leq k \, .
\end{align}
\end{Lem}

\begin{proof} 
Domain monotoniticity of eigenvalues gives for $r_1 < r_2$ and every $i$
\begin{align}	\label{e:domainMon}
	\mu_i \leq \beta_i^{r_2} \leq \beta_i^{r_1} \, .
\end{align}
This gives the first inequality in \eqr{e:thedist}.  Let  $u_1 , \dots , u_{k}$ satisfy $\cL \, u_i = - \mu_i \, u_i$ and $\| u_i \|_{L^2} = 1$.
 Corollary \ref{c:lemmamu} gives $c=c  (n,\mu_k ) > 0$ so that
\begin{align}
	\int_{\Sigma \setminus B_{r}} \left( u_i^2 + |\nabla u_i|^2 \right) \, \e^{-f} \leq \frac{c}{(r-1)^2} {\text{ for }} i = 1 , \dots , k \, .
\end{align}
Let $\psi$ be a linear cutoff function that is one on $B_r$ and zero outside of $B_{r+1}$ and 
set	$v_i = \psi \, u_i $.  The $v_i$'s are supported in $B_{r+1}$ and
we get for each $i$ that
\begin{align}	\label{e:uiandvi}
	\| u_i - v_i \|_{W^{1,2}}^2 &\leq  \int \left( (1-\psi)^2\, u_i^2 + 2\,(1-\psi)^2 |\nabla u_i|^2 + 2\, u_i^2\, |\nabla \psi|^2 \right) \, \e^{-f} \leq \frac{5c}{(r-1)^2} \, .
\end{align}
Since \eqr{e:uiandvi} implies that the $u_i$ and $v_i$ are close both in $L^2$ and for the energy, we get 
the last inequality in \eqr{e:thedist}
 for each $i$ for $r$ sufficiently large depending on $\mu_k$, $n$ and $\delta$.
\end{proof}

  \subsection{Stability of eigenvalues}
  
    In this subsection, $\Gamma^n \subset \RR^N$ is a smooth complete shrinker with finite entropy; let $\mu^{\Gamma}_i$ denote its eigenvalues.

    \begin{Def}
We will say that $\Sigma$ and $\Gamma$ are $(\epsilon , R,   C^1)$-close  if $B_R \cap \Sigma$ can be written as a normal graph of a vector field $U$ over (a subset of) $\Gamma$ and $\| U \|_{C^1} \leq \epsilon$ and, likewise, $B_R \cap \Gamma$ is a graph over $\Sigma$.
\end{Def}

\vskip1mm
The definition of $(\epsilon , R ,C^1)$-close   gives $C^1$ control in a compact set, allowing  wild differences outside of this set.  The 
next proposition shows that this is enough to get spectral stability.

    \begin{Pro}	\label{p:stable}
    Given  $k$ and $\delta > 0$, there exist $\epsilon$ and $R$ depending on $\delta , k , \Gamma$ so that if 
    a shrinker $\Sigma^n$ is $(\epsilon , R ,C^1)$-close to $\Gamma$ and $\lambda (\Sigma) < \infty$, then 
for $i \leq k$
\begin{align}
	\left| \mu^{\Gamma}_i - \mu^{\Sigma}_i \right| \leq \delta  \, .
\end{align}
    \end{Pro}

  \begin{proof}
  Lemma \ref{p:specD} gives
  $\bar{r} = \bar{r}(\mu^{\Gamma}_k , n , \delta)$ so that for $\bar{r} \leq r$
\begin{align}	\label{e:thedist2}
	\mu_i^{\Gamma} \leq \beta_i^{\Gamma,r} \leq \mu_i^{\Gamma} + \frac{\delta}{3} {\text{ for every }} i \leq k \, .
\end{align}
Moreover, given a fixed $r\leq R$, then for $\epsilon > 0$  small enough we can identify  $B_r \cap \Gamma$ and $B_r \cap \Sigma$ and, moreover, this identification is almost an isometry on $L^2$ and almost preserves the energy.  It follows that we can arrange that
  \begin{align}	\label{e:thedist3}
	\left| \beta^{\Gamma,r}_i - \beta^{\Sigma,r}_i \right| \leq \frac{\delta}{3}  \, .
\end{align}
Combining \eqr{e:thedist2}, \eqr{e:thedist3} and the fact that $\mu^{\Sigma}_i \leq  \beta^{\Sigma,r}_i$, we get an upper bound on $\mu^{\Sigma}_k$.  Hence, we can 
apply Lemma \ref{p:specD}  to get $r$ large enough that
  \begin{align}	\label{e:thedist4}
	\mu_i^{\Sigma} \leq \beta_i^{\Sigma,r} \leq \mu_i^{\Sigma} + \frac{\delta}{3} {\text{ for every }} i \leq k \, .
\end{align}
Finally, combining \eqr{e:thedist2}, \eqr{e:thedist3} and \eqr{e:thedist4} gives the proposition.
  \end{proof}

An immediate corollary is the lower semi-continuity of the spectral multiplicity (here $d(\mu)$ will denote the multiplicity of an eigenvalue $\mu$):

\begin{Cor}	\label{p:codim2}
Given  $\mu$, 
there exist  $\epsilon$, $R>0$ depending on $\mu$ and $\Gamma$ so that if a shrinker $\Sigma^n$ with $\lambda (\Sigma)<\infty$ is $(\epsilon , R ,C^1)$-close to $\Gamma$, then 
$d_{\Sigma}(\mu) \leq d_{\Gamma} (\mu)$.
\end{Cor}

 \begin{proof} 
 Theorem \ref{t:countingN} gives that the $\mu^{\Gamma}_i$'s go to infinity, so we can $\delta > 0$ so that $\Gamma$ has no eigenvalues in 
 $[\mu - 2\delta , \mu) \cup (\mu , \mu + 2\delta]$.  It follows that
 \begin{align}	\label{e:dgammamu}
 	d_{\Gamma} (\mu) = \cN_{\Gamma} (\mu + 2\delta) - \cN_{\Gamma} (\mu - 2\delta) \, .
\end{align}
Using
 Proposition \ref{p:stable} with $k = \cN_{\Gamma}(\mu + 2\delta)+1$    gives $\epsilon > 0$ and $R$ so that
 \begin{align}
 	\left| \mu_i^{\Gamma} - \mu_i^{\Sigma} \right| < \delta {\text{ for }} i \leq k \, .
 \end{align}
This implies that $\cN_{\Sigma} (\mu + \delta) \leq \cN_{\Gamma} (\mu + 2\delta)$ and 
 $\cN_{\Gamma} (\mu -2 \delta) \leq \cN_{\Sigma} (\mu -\delta)$ and, thus,
 \begin{align}
 	d_{\Sigma} (\mu) \leq \cN_{\Sigma} (\mu + \delta) - \cN_{\Sigma} (\mu - \delta) \leq  \cN_{\Gamma} (\mu + 2\delta)
	-  \cN_{\Gamma} (\mu - 2\delta) \, .
 \end{align}
 The corollary follows by combining this with \eqr{e:dgammamu}.
\end{proof}

In   \cite{dCW}, do Carmo-Wallach construct families of minimal submanifolds of the sphere, each  isometric to the same round sphere, 
 generalizing   results of Calabi, \cite{Ca}.  The boundary immersions of the families in \cite{dCW} lie   in a lower-dimensional affine  space.   Obviously, they have the same
  volume and, since they are contained in spheres, also the same entropy.  Thus,  the number of linearly independent coordinate functions can vary along a family.

\subsection{Shrinking curves}

In \cite{AL}, Abresch-Langer classified closed shrinking curves in $\RR^2$.  The only embedded one is the circle $\SS^1_{\sqrt{2}}$.  There are immersed solutions $\gamma_{m,\ell}$ with $\frac{1}{2} < \frac{m}{\ell} < \frac{\sqrt{2}}{2}$ where $m$ is the rotation index and $\ell$ is the number of periods of its curvature function.  Moreover, each $\gamma_{m,\ell}$ is convex and has self-intersections, so
$\lambda (\gamma_{m,\ell}) > 2$.

\begin{Lem}	\label{l:curvelambda}
If $\gamma^1 \subset \RR^N$ is a complete immersed shrinker and $\lambda (\gamma) < \infty$, then $\gamma$ is a rotation of either $\RR$, $\SS^1_{\sqrt{2}}$, or one of the $\gamma_{m,\ell}$'s.
\end{Lem}

\begin{proof}
Suppose first that $\gamma \subset \RR^2$ with unit normal $\nn$ and let $k =  \frac{1}{2} \langle x , \nn \rangle$ be its geodesic curvature.  By \cite{AL}, the quantity $k \, \e^{ - \frac{|x|^2}{4}}$ is constant (this follows from differentiating the equation $k =   \frac{1}{2} \langle x , \nn \rangle$).  If $k$ ever vanishes, then it is identically zero and $\gamma = \RR$.  Otherwise, we can assume that $k>0$ and there is a constant $c> 0$ with
\begin{align}
	c\, \e^{ \frac{|x|^2}{4} } = k \leq \frac{|x|}{2} \, .
\end{align}
  It follows that $k \geq c > 0$  and $|x|$, and thus also $k$, 
  are bounded from above.
  Thus, $\gamma$ is a convex curve in a bounded region.  Since $\lambda (\gamma) < \infty$ and $\cL \, |x|^2 = 2- |x|^2$, we have
  \begin{align}
  	4\, c^2 \, \int_{\gamma}    \e^{   \frac{|x|^2}{4} }= 4\, \int_{\gamma}  k^2 \, \e^{ - \frac{|x|^2}{4} }  \leq \int_{\gamma}  |x|^2 \, \e^{ - \frac{|x|^2}{4} } = 2 \, \int_{\gamma}  \e^{ - \frac{|x|^2}{4} }  = 2 \, \lambda (\gamma) \, .
  \end{align}
Therefore, $\gamma$ has finite length and must be one of the Abresch-Langer curves.
Finally, by uniqueness for ODE's, these are also the only shrinking curves in $\RR^N$ (up to rotation).  
\end{proof}

The next lemma shows that coordinate functions generate the entire $\frac{1}{2}$-eigenspace on the product of an Abresch-Langer curve with $\RR^{n-1}$:

\begin{Lem}	\label{l:spec}
For any $n$, $m$ and $\ell$, we have $d_{\gamma_{m,\ell} \times \RR^{n-1}} (\frac{1}{2})= n+1$.
\end{Lem}

\begin{proof}
Set $\gamma = \gamma_{m,\ell}$  and $\Gamma = \gamma_{m,\ell} \times \RR^{n-1}$.
Following lemma $3.26$ in \cite{CM7}, let $y_i$ be coordinates on $\RR^{n-1}$ so that $\cL$ splits as
\begin{align}
	\cL = \cL_{\gamma} + \cL_{y} \, , 
\end{align}
where $\cL_{\gamma}$ is the drift operator on $\gamma \subset \RR^2$ and $\cL_y$ is the drift operator on $\RR^{n-1}$.
  Suppose that $u$ is an $L^2$ and, thus also $W^{1,2}$, function on $\Gamma$ satisfying $\cL \, u = - \frac{1}{2} \, u$.
Set $u_i = \frac{\partial u}{\partial y_i}$ and observe that
$
	\cL \, u_i = 0 $.  
It follows that $u_i$ is constant.  Since this holds for each $i$,  $u= \sum_i a_i \, y_i + g$ where the $a_i$'s are constants and $g$ is a function on $\gamma$.
Consequently, to prove the lemma, we must show that the two coordinate functions generate the entire $\frac{1}{2}$-eigenspace on $\gamma \subset \RR^2$.  However, this follows immediately from uniqueness for the second order ODE $\cL_{\gamma} \, g = - \frac{1}{2} \, g$.
\end{proof}
  
\subsection{Rigidity of spheres and cylinders}

\begin{Cor}  \label{l:spechyp2}
Given $k < n$, there exists $R> 2n$ so if 
  $\Sigma^n\subset \RR^N$ is a shrinker with $\lambda (\Sigma) < \infty$ and $B_R \cap \Sigma$ is the graph of a vector field $U$ over   
  $\SS^k_{\sqrt{2k}}\times \RR^{n-k}$  with $\| U \|_{C^1} < 1/R$, then there is a $(n+1)$-dimensional Euclidean space $\cW$ so that $\Sigma \subset \cW$.
\end{Cor}

\begin{proof} 
 On the cylinder $\SS^k_{\sqrt{2k}}\times \RR^{n-k}$  (see, e.g., section $3$ in \cite{CM7}), the low eigenvalues of $\cL$ are $\mu_0 = 0$, given by the constants, 
  $\frac{1}{2}$ with multiplicity $n+1$, and then there is a gap to
  $1$.
More precisely, it follows from lemma $3.26$ in \cite{CM7} (and its proof{\footnote{Lemma $3.26$ in \cite{CM7} deals with eigenvalue $1$; obvious modifications give eigenvalue $\frac{1}{2}$ as well.}}) that:
\begin{itemize}
\item If $u \in W^{1,2}$ has $\cL \,  u = - \frac{1}{2} \, u$, then $u = f(\theta) + \sum a_j \, y_j$ where $f$ is a $\frac{1}{2}$-eigenfunction on $\SS^k$, $a_j$ are constants, and $y_j$ are  coordinate functions on the axis $\RR^{n-k}$.
\end{itemize}
It follows that $d_{ \SS^k_{\sqrt{2k}}\times \RR^{n-k} } ( \frac{1}{2}) = n+1$ and Corollary \ref{p:codim2} with $\mu = \frac{1}{2}$ completes the proof.
\end{proof}

\begin{proof}[Proof of 
Theorem
\ref{t:rigidity}] As   long as $R$ is large enough, 
Corollary \ref{l:spechyp2} gives  that $\Sigma$ is a hyperplane in some affine $(n+1)$-space.   By assumption, it is also close to a cylinder.  We will apply 
the main rigidity theorem for hypersurfaces from \cite{CIM}, but first need to establish a uniform  bound for $\int_{\Sigma} \e^{-f}$ (which gives an entropy bound).  However, this follows from the closeness to a cylinder in $B_R$, for $R$ large, since $\int_{\Sigma} (|x|^2 - 2n) \, \e^{-f} = 0$.

In the case of a sphere (i.e., $k=n$), we   can argue directly without \cite{CIM}.  Namely, since $\Sigma$ is a shrinker, we have $\cL \, (|x|^2-2n) = -  (|x|^2-2n)$.  However, 
  $1$ is not in the spectrum for $\SS^n_{\sqrt{2n}}$, so Corollary \ref{p:codim2} gives
 $1$  is also not in the spectrum of   $\Sigma$ for $\Sigma$ sufficiently close.  It follows that $|x|^2 - 2n \equiv 0$ and
  $\Sigma \subset \partial B_{\sqrt{2n}}$.
\end{proof}
 
 We get the corresponding statement for products with the Abresch-Langer curves $\gamma_{m,\ell}$:

\begin{Cor}  \label{l:spechyp2AL}
Given $m , \ell , n$, there exists $R> 2n$ so if 
  $\Sigma^n\subset \RR^N$ is a shrinker with $\lambda (\Sigma) < \infty$ and $B_R \cap \Sigma$ is a parameterized graph of a vector field $U$ over  
  $\gamma_{m,\ell} \times \RR^{n-1}$  with $\| U \|_{C^1} < 1/R$, then there is a $(n+1)$-dimensional Euclidean space $\cW$ so that $\Sigma \subset \cW$.
\end{Cor}

\begin{proof}
This follows by combining Lemma \ref{l:spec}  and Corollary \ref{p:codim2} with $\mu = \frac{1}{2}$.
\end{proof}

The rigidity should also hold for $\gamma_{m,\ell} \times \RR^{n-1}$ by combining Corollary \ref{l:spechyp2AL} with a modification of \cite{CIM}.   See \cite{CM11} for rigidity of cylinders in Ricci flow.

    \section{Sharp bounds for codimension}
    
    In this section, $M_t^n \subset \RR^N$ is an ancient MCF with    $\lambda (M_t) \leq \lambda_0$ and $\phi = \bH + \frac{x^{\perp}}{2t}$.

  \subsection{Preserving orthogonality}
  
  The next lemma shows that a caloric function that integrates to zero on one time slice must remain nearly orthogonal to constants, with the error bounded by the change in the Gaussian area $I_1(t)$.

   \begin{Lem}	\label{l:diffcont}
If   $(\partial_t - \Delta)\, u = 0$ and $J_{t_1} (u,1) =0$ for some $t_1 < 0$, then for any $t_2 \in [t_1 , 0)$
\begin{align}	\label{e:diffcont1}
  	\left|  J_{t_2} (u ,1 ) \right|^2 & \leq    I_{u}(t_1)  \,  \left| I_1 (t_1) - I_{1} (t_2)
	\right|     \, .
  \end{align}
   \end{Lem}
  
  \begin{proof}
  Given $t \in [t_1 , t_2]$,   \eqr{e:eckerH}  gives the derivative   
  \begin{align}
  	\frac{d}{dt} \, J_{t} (u , 1) = - \left( -4\pi \, t \right)^{ - \frac{n}{2} } \, \int_{M_t} u \, |\phi|^2 \, \e^{ \frac{ |x|^2}{4t} } \, .
  \end{align}
  Since $J_{t_1}(u,1)=0$, integrating  from $t_1$ to $t_2$ and using the Cauchy-Schwarz inequality gives
  \begin{align}
  	\left| J_{t_2} (u,1) \right|^2 &\leq  \left\{ \int_{t_1}^{t_2} \left( -4\pi \, t \right)^{ - \frac{n}{2} } \, \int_{M_t} u^2 \, |\phi|^2 \, \e^{ \frac{ |x|^2}{4t} } \right\}  \, 
		\left\{ \int_{t_1}^{t_2}  \left( -4\pi \, t \right)^{ - \frac{n}{2} } \, \int_{M_t}   |\phi|^2 \, \e^{ \frac{ |x|^2}{4t} } \right\} \, .
  \end{align}
By  \eqr{e:eckerH}, the last integral on the right is bounded by  $ \left| I_1 (t_1) - I_{1} (t_2)
	\right| $.    Similarly, by Lemma \ref{l:iprimelem}, the first integral is bounded by $\left| \int_{t_1}^{t_2} I_u'(t) \right| \leq I_{u}(t_1) - I_u(t_2)$.
	\end{proof}

	In the next two lemmas, $V \in \SS^{N-1}$ is a unit vector and $v$ the linear function $v(x) = \langle x , V \rangle$.   Let $\cL_t = \Delta + \frac{1}{2t} \, \nabla_{x^T}$ be the drift operator that is symmetric for $\e^{ \frac{|x|^2}{4t}}$.  We will use that
	\begin{align}	\label{e:cLtV}
		\cL_t \, v = \dv_{M_t} V^T + \frac{1}{2t} \langle x^T , V \rangle = -  \langle V , \bH \rangle + \langle \frac{x^T}{2t} , V \rangle
		= \frac{v}{2t} - \langle \phi , V \rangle \, .
	\end{align}
	
	\begin{Lem}	\label{l:vV1}
	We get for   $t_1 < t_2 <0$ that
	\begin{align}
		\int_{t_1}^{t_2}  \left( -4\pi \, t \right)^{ - \frac{n}{2} } \, \int_{M_t}     \frac{v^2}{-t} \,  |\phi|^2  \, \e^{ \frac{ |x|^2}{4t} } \, dt & \leq  \left|
		  \frac{ I_v(t_1)}{t_1}  -  \frac{I_v(t_2)}{t_2}
		 \right|+ C_n \, \sqrt{\lambda_0}  \,\left( \frac{t_1}{t_2} \right)^{ \frac{1}{2}}  \,  \left| I_1 (t_1) - I_{1} (t_2)
	\right|^{ \frac{1}{2} }
		  \, . \notag
	\end{align}
	\end{Lem}
	 
	 \begin{proof}
	 Using \eqr{e:eckerH} for $\frac{v^2}{-t}$, then the divergence theorem and \eqr{e:cLtV}  gives
	\begin{align}	\label{e:startvV1}
		\frac{d}{dt} \, \left( \frac{I_v(t)}{-t} \right) &=    \left( -4\pi \, t \right)^{ - \frac{n}{2} } \, \int_{M_t} \left\{   \frac{v^2}{t^2} - \frac{2|\nabla v|^2}{-t} - \frac{v^2}{-t} \,  |\phi|^2  \right\}\, \e^{ \frac{ |x|^2}{4t} } \notag \\
		&=    \left( -4\pi \, t \right)^{ - \frac{n}{2} } \, \int_{M_t} \left\{    - \frac{2v \, \langle \phi , V \rangle }{-t} - \frac{v^2}{-t} \,  |\phi|^2  \right\}\, \e^{ \frac{ |x|^2}{4t} }  \, .
	\end{align}
	Using that $|V| =1$, we get the absorbing inequality for any $\epsilon > 0$
	\begin{align}	\label{e:7point8}
		\left|    \frac{2v \, \langle \phi , V \rangle }{-t} \right| \leq \epsilon \, \frac{v^2}{t^2} + \frac{1}{\epsilon} \, |\phi|^2 \, .
	\end{align}
Using \eqr{e:7point8} in \eqr{e:startvV1}, we get that
	\begin{align}
		 \left( -4\pi \, t \right)^{ - \frac{n}{2} } \, \int_{M_t}   \frac{v^2}{-t} \,  |\phi|^2  \, \e^{ \frac{ |x|^2}{4t} } \leq -\frac{d}{dt} \, \left( \frac{I_v(t)}{-t} \right)  + \epsilon \, \frac{I_v(t)}{t^2} + \frac{1}{\epsilon} \,  \left( -4\pi \, t \right)^{ - \frac{n}{2} } \, \int_{M_t}     {|\phi|^2} \,  \e^{ \frac{ |x|^2}{4t} } \, .
	\end{align}
	Integrating this from $t_1$ to $t_2$ and using the monotonicity of $I_v$ to bound the second term on the right and Huisken's monotonicity on the last term gives
	\begin{align}
		\int_{t_1}^{t_2}  \left( -4\pi \, t \right)^{ - \frac{n}{2} } \, \int_{M_t}     \frac{v^2}{-t} \,  |\phi|^2  \, \e^{ \frac{ |x|^2}{4t} } \, dt & \leq  \left|
		  \frac{ I_v(t_1)}{t_1}  -  \frac{I_v(t_2)}{t_2}
		 \right|+ \epsilon \, \frac{I_v (t_1)}{-t_2} +   \frac{\left| I_1 (t_1) - I_{1} (t_2)
	\right|}{\epsilon}
		  \, . \notag
	\end{align}
	The lemma follows by using  that  $I_v (t) \leq -C_n \, \lambda_0 \, t$ (cf. \eqr{e:Iupolyd}) and  optimizing $\epsilon$.
    \end{proof}
    
    The next lemma shows that the inner product of a caloric function with a fixed linear function grows approximately linearly in $t$.

\begin{Lem}	\label{l:vV2}
If $u_t = \Delta u$ and $t_1 < t_2 < 0$, then for any $\epsilon_1 \in (0,1/2)$ and all $t \in [t_1 , t_2]$
	\begin{align}	 
		& \sqrt{-t_2} \, \left|   \frac{ J_{t_1} (u , v) }{t_1}  - \frac{J_{t} (u , v)}{t}\right|  \leq \frac{5}{2} \, \epsilon_1  \,  I_{u}(t_1)  + \frac{1}{\epsilon_1} \, \left| I_1 (t_1) - I_{1} (t_2)
	\right|   
	\notag \\
	&\qquad  +  \frac{1}{2\, \epsilon_1} \left\{
	\, \left|
		  \frac{ I_v(t_1)}{t_1}  -  \frac{I_v(t_2)}{t_2}
		 \right|+ C_n \, \sqrt{\lambda_0}  \,\left( \frac{t_1}{t_2} \right)^{ \frac{1}{2}}  \,  \left| I_1 (t_1) - I_{1} (t_2)
	\right|^{ \frac{1}{2} }	\right\}   \, .  \notag
	\end{align}
	\end{Lem}

	\begin{proof}
	Since $u$ and $v$ satisfy the heat equation, 
	applying \eqr{e:eckerH}    to $v \, u$   gives
	\begin{align}	\label{e:startdiff}
		\frac{d}{dt} \, J_{t} (u , v) = - \left( -4\pi \, t \right)^{ - \frac{n}{2} } \, \int_{M_t} \left\{2\, \langle \nabla u , \nabla v \rangle +  u \,v |\phi|^2  \right\}\, \e^{ \frac{ |x|^2}{4t} } \, .
	\end{align}
	Using  \eqr{e:cLtV} and the divergence theorem on the first term in \eqr{e:startdiff} gives
	\begin{align}	 
		\frac{d}{dt} \, J_{t} (u, v) &=  \left( -4\pi \, t \right)^{ - \frac{n}{2} } \, \int_{M_t} \left\{2\,   u \, \cL_t \, v   -  u\,v |\phi|^2  \right\}\, \e^{ \frac{ |x|^2}{4t} } \notag \\
		   &= \frac{1}{t} \, J_{t} (u , v) - \left( -4\pi \, t \right)^{ - \frac{n}{2} } \, \int_{M_t} \left\{2\,   u \, \langle \phi , V \rangle   +  u \,v |\phi|^2  \right\}\, \e^{ \frac{ |x|^2}{4t} } 
		   \, .
	\end{align}
Using absorbing inequalities on the last integral, we get for any $\epsilon_1  > 0$ that
	\begin{align}	
		\left| \frac{d}{dt} \, \frac{J_{t} (u, v)}{t} \right| &\leq \epsilon_1 \, \frac{I_{u} (t)}{ (-t)^{3/2}} + 
		 \left( -4\pi \, t \right)^{ - \frac{n}{2} }   \, \int_{M_t} \left\{ \frac{ | \phi|^2}{\epsilon_1 \, \sqrt{-t}}   +  \left( \frac{\epsilon_1\, u^2}{2 \sqrt{-t}}   + \frac{v^2}{2\epsilon_1\, (-t)^{3/2}}  \right) |\phi|^2  \right\}\, \e^{ \frac{ |x|^2}{4t} } 
		   \, .  \notag
	\end{align}
	Integrating  in $t$, using the monotonicity of $I_{u}$ on the first term, \eqr{e:eckerH}  on the  next two terms, and Lemma \ref{l:vV1} on the last term gives for any $t \in [t_1 , t_2]$ that
	\begin{align}	 
		& \sqrt{-t_2} \, \left|   \frac{ J_{t_1} (u , v) }{t_1}  - \frac{J_{t} (u , v)}{t}\right|  \leq 2\, \epsilon_1  \,  I_{u}(t_1)  + \frac{1}{\epsilon_1} \, \left| I_1 (t_1) - I_{1} (t_2)
	\right|  + \frac{\epsilon_1}{2 } \, \left| I_{u} (t_1) - I_{u} (t_2) \right| 
	\notag \\
	&\qquad  +  \frac{1}{2\, \epsilon_1} \left\{
	\, \left|
		  \frac{ I_v(t_1)}{t_1}  -  \frac{I_v(t_2)}{t_2}
		 \right|+ C_n \, \sqrt{ \lambda_0}  \,\left( \frac{t_1}{t_2} \right)^{ \frac{1}{2}}  \,  \left| I_1 (t_1) - I_{1} (t_2)
	\right|^{ \frac{1}{2} }	\right\}   \, .  \notag
	\end{align}
  \end{proof}

\subsection{Poincar\'e inequalities}

The first eigenvalue on a cylinder is $\frac{1}{2}$, with the eigenspace spanned by the $(n+1)$ coordinate functions, and the next eigenvalue is $1$.  The next lemma gives corresponding Poincar\'e inequalities
for submanifolds    close to a cylinder on a fixed large set.  This requires that   the function satisfies a  ``localization inequality'' (cf. Lemma \ref{l:localizeMt}):
\begin{align}	\label{e:concinq}
	 \left( -4\pi t \right)^{- \frac{n}{2} } \, \int_{\Gamma}\left(  \frac{|x|^2}{-t} \, u^2 -t\, |\nabla u|^2 \right)  \, \e^{ \frac{|x|^2}{4t}} \leq C_0 \, \left( -4\pi t \right)^{- \frac{n}{2} } \, \int_{\Gamma} u^2 \, \e^{ \frac{|x|^2}{4t}} < \infty \, . 
\end{align}
In the next lemma, $\Gamma^n \subset \RR^N$ is a submanifold with $\lambda (\Gamma) \leq \lambda_0 < \infty$ and $B_{R_{\mu}} \cap \frac{\Gamma}{\sqrt{-t}}$ is a $C^1$ graph over a cylinder with norm at most $\epsilon_{\mu}$ and  $t< 0$ is a constant.

\begin{Lem}	\label{l:poin1}
Given $C_0$ and $\mu > 0$, there exists $\epsilon_{\mu} > 0$ and $R_{\mu} >0$ so that if 
  $u$ is a $W^{1,2}$ function satisfying $\int_{\Gamma} u \, \e^{ \frac{|x|^2}{4t}} = 0$ and \eqr{e:concinq}, 
then  
\begin{align}	\label{e:PoA}
	(1-\mu) \, \left( -4\pi t \right)^{- \frac{n}{2} } \, \int_{\Gamma} u^2 \, \e^{ \frac{|x|^2}{4t}} \leq -2\, t\, \left( -4\pi t \right)^{- \frac{n}{2} } \,  \int_{\Gamma} |\nabla u|^2 \, \e^{ \frac{|x|^2}{4t}} \, .
\end{align}

If in addition $\int_{\Gamma} u \, x_i \, \e^{ \frac{|x|^2}{4t}} = 0$ for  the coordinates $x_1  , \dots , x_{n+1}$   on the cylinder, then
\begin{align}	\label{e:PoB}
	(1-\mu) \, \left( -4\pi t \right)^{- \frac{n}{2} } \, \int_{\Gamma} u^2 \, \e^{ \frac{|x|^2}{4t}} \leq  -t \, \left( -4\pi t \right)^{- \frac{n}{2} } \, \int_{\Gamma} |\nabla u|^2 \, \e^{ \frac{|x|^2}{4t}} \, .
\end{align}
\end{Lem}

\begin{proof}
Since the statement is scale-invariant, we can assume that $t=-1$.  To shorten notation, let
$\fint$ denote the Gaussian  integral
\begin{align}
	\fint  w \equiv  \left( 4\pi \right)^{ - \frac{n}{2} } \, \int w \, \e^{-f} \, .
\end{align}
Let $L$ be a  large integer to be chosen and   choose   $R \in \{L, L+1, \dots , 2L-1 \}$ with
\begin{align}	 
	  \fint_{B_{R+1} \cap \Gamma \setminus B_R} \, |\nabla u|^2   \leq 
	  \frac{ 1}{L}    \, \fint_{B_{2L} \cap \Gamma \setminus B_L} |\nabla u|^2  \leq
	  	 \frac{ 1}{L}   \, \fint_{\Gamma} |\nabla u|^2 \, .   
\end{align}
Combining this with the localization inequality \eqr{e:concinq} gives
\begin{align}	\label{e:Po1}
	 \fint_{B_{R+1} \cap \Gamma \setminus B_R} \, |\nabla u|^2   \leq \frac{C_0}{L} \,     \fint_{\Gamma} u^2\, .
\end{align}
Let $\psi$ be a linear cutoff function from $R$ to $R+1$ and define $w = \psi \, u$.  The localization inequality \eqr{e:concinq} gives
\begin{align}	\label{e:Po2}
	 \fint_{\Gamma}   |u - w|^2   \leq    \fint_{\Gamma \setminus B_{R}}   u^2    \leq \frac{C_0}{R^2} \,   \fint_{\Gamma} u^2   \, .
\end{align}
As long as $B_{2R} \cap \Gamma$ is a small $C^1$ graph over the cylinder $\Sigma$, we can transplant the function $w$ to a function $\bar{w}$ on $\Sigma$ which is supported inside $B_{2R}$.  Moreover, the distortion of the measure and the gradient are as small as we want if we make $\epsilon_{\mu}$ small enough.  In particular, there is a continuous function $\eta (R , \epsilon_{\mu})$ with $\eta (R , 0) =0$ so that
\begin{align}
	\left| \fint_{\Gamma} w^2  -  \fint_{\Sigma} \bar{w}^2   \right| &\leq \eta \,  \fint_{\Gamma} w^2   \, ,  \label{e:Po4} \\
	\left|   \fint_{\Gamma} |\nabla w|^2  -    \fint_{\Sigma} |\nabla \bar{w}|^2  \right| &\leq \eta \,    \fint_{\Gamma} |\nabla w| ^2   \, , \label{e:Po5} \\
	\left|   \fint_{\Gamma} w   -   \fint_{\Sigma} \bar{w}  \right|^2 &\leq \eta \,   \fint_{\Gamma} w^2  
	\, .  \label{e:Po6}
\end{align}
The first non-zero eigenvalue of the cylinder $\Sigma$ is $\frac{1}{2}$, so $\bar{v}$ satisfies
\begin{align}
	  \fint_{\Sigma} \bar{w}^2  &\leq  \frac{ \left(\fint_{\Sigma} \bar{w}  \right)^2}{\lambda (\Sigma)  } + 2 \, \fint_{\Sigma} |\nabla \bar{w} |^2  
	  \leq    \left( \fint_{\Sigma} \bar{w} \right)^2 + 2 \, \fint_{\Sigma} |\nabla \bar{w} |^2  
	\, .
\end{align}
Combining this with  \eqr{e:Po4}, \eqr{e:Po5}, \eqr{e:Po6} and the squared triangle inequality gives
\begin{align}
	(1-\eta) \,  \fint_{\Gamma} w^2  &\leq    \fint_{\Sigma} \bar{w}^2 \leq  2 \,    \fint_{\Sigma} |\nabla \bar{w} |^2   +   \left(   \fint_{\Sigma} \bar{w} \right)^2  \notag \\
	&\leq 2 (1+\eta) \,  \fint_{\Gamma} |\nabla w|^2    + 2  \, \left(\eta \,    \fint_{\Gamma} w^2  + \left[   \fint_{\Gamma} w  \right]^2  \right)    \, .
\end{align}
Absorbing the middle term, using $\fint_{\Gamma} u  = 0$, \eqr{e:Po2}  and the Cauchy-Schwarz inequality gives
\begin{align}
	 (1-3\, \eta) \,   \fint_{\Gamma} w^2   & \leq 2 (1+\eta) \,  \fint_{\Gamma} |\nabla w|^2 +  2\,   \left[ \fint_{\Gamma} (u-w) \right]^2    \notag \\
	 &\leq 2 (1+\eta) \fint_{\Gamma} |\nabla w|^2  +   \frac{C }{R^2}\,  \fint_{\Gamma} u^2   \, .  
\end{align}
Using  \eqr{e:Po2}  to bound  $\| u \|_{L^2}$ by $\| w \|_{L^2}$  and   splitting   $\fint_{\Gamma}|\nabla w|^2 $  into   inner and outer parts  
\begin{align}
	\left(1- \frac{C_0}{R^2} \right) & (1-3\eta)  \fint_{\Gamma} u^2   \leq (1-3\eta) \,    \fint_{\Gamma} w^2   \notag \\
	& \leq 2 (1+\eta) \, \fint_{B_R \cap \Gamma} |\nabla w|^2  + 
	   \frac{C}{R^2} \,  \fint_{\Gamma} u^2   + 2(1+\eta) \, \fint_{\Gamma \setminus B_R} |\nabla w|^2    \, . 
\end{align}
  The first term is bounded by $2 (1+\eta) \, \fint_{  \Gamma} |\nabla u|^2 $ since $u=w$ on $B_R \cap \Gamma$.  The second can be absorbed on the left.  
 For the last, we use \eqr{e:Po1} and  \eqr{e:Po2}  to get
\begin{align}
	\fint_{\Gamma \setminus B_R} |\nabla w|^2  \leq 2 \,\fint_{\Gamma \setminus B_R}(u^2 +  |\nabla u|^2)   \leq 2 \left( \frac{C_0}{R^2} + \frac{C_0}{L} \right) \, \fint_{\Gamma} u^2     \, .   
\end{align}
Choosing $L$ large and then taking $\eta$ small enough, this gives \eqr{e:PoA}.

Suppose   in addition that  $\fint_{\Gamma} u \, x_i  = 0$ for  the coordinates $x_1  , \dots , x_{n+1}$   on $\Sigma$. These $x_i$'s are a   basis for the first non-zero eigenspace of $\Sigma$ and the next
eigenvalue is $1$.  Thus,   if $\zeta$ is a function on $\Sigma$ that integrates to zero against $1$ and 
 $x_1, \dots , x_{n+1}$, then $\fint_{\Sigma} \zeta^2   \leq \fint_{\Sigma} |\nabla \zeta|^2$.  For each of these $x_i$'s,  using that the support of $w$ is a graph over $\Sigma$   gives
 \begin{align}
 	\left( \fint_{\Gamma} w \, x_i   -\fint_{\Sigma} \bar{w} \, x_i   \right)^2 \leq \eta \, \fint_{\Gamma} w^2 \, .
 \end{align}
 We can now argue as above to get \eqr{e:PoB}.
\end{proof}

   \subsection{Lower bounds for growth}
   
   The next lemma shows that any caloric function that has integral zero on a slice must grow at least linearly if the flow is close to cylindrical.
   
   \begin{Lem}		\label{l:atleastlinear}
   Given $\mu \in (0, 1/2)$, there exist $\epsilon_{\mu} > 0$ and $R_{\mu} > 0$ so that if 
   \begin{itemize}
   \item $B_{R_{\mu}} \cap  \frac{ M_t }{\sqrt{-t}} $ is an $\epsilon_{\mu}$ $C^1$-graph over a cylinder for $t_1 \leq t \leq t_2 < 0$,
   \item $u_t = \Delta u$, $I_u (t_1)= 1$, and $J_{t_1} (u,1)= 0$,
   \end{itemize}
   then $I_{u} (t_2) \leq \left( \frac{t_1}{t_2} \right)^{\mu-1} + 2 \,  \left| I_1 (t_1) - I_1(t_2) \right|  $.
   \end{Lem}
   
   \begin{proof}
   Set $\kappa =  \left| I_1 (t_1) - I_{1} (t_2)
	\right|$.  We are done if $I_{u}(t_2) \leq 2\, \kappa$, so we    can assume 
	\begin{align}	\label{e:kappaIu}
		2\, \kappa< I_{u}(t_2) \leq I_u(t) {\text{ for all }} t \in [t_1 , t_2] \, .
	\end{align}
   When \eqr{e:kappaIu} holds, we will prove that $I_u$ satisfies the differential inequality
\begin{align}	\label{e:theguythatholds}
	\left( (-t)^{\mu-1} \, I_u \right)' = (-t)^{\mu-1} \,  \left(  I_u' + (\mu-1) \frac{I_u}{t} \right)  \leq    \kappa \, (-t)^{\mu - 2 }  \, .
\end{align}
Once we have   \eqr{e:theguythatholds}, then integrating from $t_1$ to $t_2$ gives
\begin{align}
	    (-t_2)^{\mu - 1} \, I_u(t_2) - (-t_1)^{\mu-1} \, I_u (t_1) \leq   \kappa \, \int_{t_1}^{t_2} (-t)^{\mu -2} \, dt \leq \frac{ \kappa  \,  \, (-t_2)^{ \mu - 1}  }{1-\mu} \, .
\end{align}
The lemma follows from this since $I_u (t_1) = 1$ and $\mu \in (0,1/2)$.
 
  It remains to prove \eqr{e:theguythatholds}.    If $I_u'(t) <   \frac{I_u}{t}$, then  \eqr{e:theguythatholds} holds.  Hence, suppose  that  
 \begin{align}	\label{e:badt}
 	I_u'(t) \geq   \frac{I_u}{t} \, .
 \end{align}
In this case $t\, I_u'  (t)\leq I_u (t)$ and  \eqr{e:Ipri} gives
    \begin{align}	 	\label{e:instead}
 	-t\, \left( -4\pi \, t \right)^{ - \frac{n}{2} } \, \int_{M_t} \left\{ 2 \, |\nabla u|^2 + u^2 \, |\phi |^2 \right\}  \, \e^{ \frac{|x|^2}{4t} } = t \, I'_u(t)  \leq I_u(t)  \, .
 \end{align}
 Combining \eqr{e:instead} with   Lemma \ref{l:localizeMt}, 
gives  
 \begin{align}	\label{e:locali2}
 	 \frac{\left( -4\pi \, t \right)^{ - \frac{n}{2} }}{-t} \int_{M_t} |x|^2 \, u^2 \,  \e^{   \frac{|x|^2}{4t} } & \leq 4n \, I_u(t) - 4t \, \left( -4\pi \, t \right)^{ - \frac{n}{2} } \int_{M_t}  \left( 4 \, |\nabla u|^2 + u^2 \, |\phi|^2 \right)  \e^{   \frac{|x|^2}{4t} } \notag \\
	 &\leq (4n+8) \, I_u(t) \, .
 \end{align}
  For each $t \in [t_1 , t_2]$, Lemma \ref{l:diffcont}
gives that
\begin{align}	 \label{e:fromdiffcont}
  	\left|  J_{t} (u ,1 ) \right|^2 & \leq  \kappa      \, .
  \end{align}
 The function $v = u - \frac{J_t (u,1)}{I_1 (t)}$ integrates to zero on $M_t$ and, using \eqr{e:fromdiffcont}
 and \eqr{e:kappaIu}, 
 \begin{align}
 	I_u(t) &= I_v (t) + \frac{J_t^2 (u,1)}{I_1 (t)} \leq I_v (t) + \kappa <   I_v(t) + \frac{1}{2} \, I_u (t)  \, .
\end{align}  
 From this, we conclude that
 \begin{align}	\label{e:IuIvt}
 	I_u (t) \leq 2 \, I_v (t) \, .
 \end{align}
We will show that $v$ satisfies the localization inequality \eqr{e:concinq}.  The energy bound in \eqr{e:concinq} follows from \eqr{e:instead} and \eqr{e:IuIvt} since $|\nabla v|^2 = |\nabla u|^2$.
  The squared triangle inequality for $v$,  \eqr{e:locali2}, and using the entropy bound on $M_t$ to bound $\left( -4\pi \, t \right)^{ - \frac{n}{2} } \int_{M_t} \frac{|x|^2}{-t}  \,  
  \e^{   \frac{|x|^2}{4t}}$ give
\begin{align}
	 \frac{\left( -4\pi \, t \right)^{ - \frac{n}{2} }}{-t} \int_{M_t} |x|^2 \, v^2 \,  \e^{   \frac{|x|^2}{4t} } & \leq 2\,  \frac{\left( -4\pi \, t \right)^{ - \frac{n}{2} }}{-t} \int_{M_t} |x|^2 \, (u^2 + \kappa) \,  \e^{   \frac{|x|^2}{4t} } \notag \\
	 &\leq  2\,(4n+8) \, I_u(t) + C \, \kappa  \, .
 \end{align}
 Using \eqr{e:kappaIu} and \eqr{e:IuIvt}, we get the  remaining estimate for  \eqr{e:concinq}.
Thus,   if $M_t /\sqrt{-t}$ is sufficiently close to cylindrical,   Lemma \ref{l:poin1}, \eqr{e:fromdiffcont}  and the equality in  \eqr{e:instead} give
\begin{align}	\label{e:diffIneqbadt}
	(1-\mu) \, I_u(t)  \leq     \frac{J^2_t (u,1)  }{I_1(t)}
	- 2\, t\,   \left( -4\pi \, t \right)^{ - \frac{n}{2} }  \int_{M_t}  |\nabla u|^2    \, \e^{ \frac{|x|^2}{4t} }  \leq  \kappa +t \, I_u' (t) 
	\, .
\end{align}
This gives \eqr{e:theguythatholds} in the remaining case \eqr{e:badt}, completing the proof.
   \end{proof}

 \subsection{Projecting orthogonally to linear functions}
 Let $x_1 , \dots , x_{n+1}$ be  the coordinates on the cylinder
  $\Sigma=\SS^k_{\sqrt{2k}} \times \RR^{n-k} \subset \RR^{n+1} \subset \RR^N$.    Given   $u$ on $M_t$, define  $\zeta = \zeta (t) \in \RR^{n+2}$ by
    \begin{align}	\label{e:hereiszeta}
    	   \zeta_0 = J_t (u,1) {\text{ and }} \zeta_i = \frac{ J_t (u , x_i)}{\sqrt{-t}} {\text{ for }} i = 1 , \dots , n+1 \, .
\end{align}
The function $\frac{x_i}{\sqrt{-t}}$ is normalized to be roughly unit size on $M_t$ and 
 constant size on a self-shrinking flow.  Let  $a=a(t) \in \RR^{n+2}$ be coefficients so that
  $v= u - a_0 - \sum_{i=1}^{n+1} a_i \frac{x_i}{\sqrt{-t}}$ is  the $J_t$-projection  of $u$ orthogonally to $\{ 1 , \frac{x_1}{\sqrt{-t}} , \dots , \frac{x_{n+1}}{\sqrt{-t}} \}$.

 If $\zeta =0$, then $a=0$ and $v=u$.  The next lemma shows that $u$ and $v$ are close if $\zeta$ is small.

\begin{Lem}	\label{l:projction}
There exist $ R_0 ,  \epsilon_0 > 0$   so that if   $B_{R_0} \cap \frac{M_t}{\sqrt{-t}}$ is an $\epsilon_0$ $C^1$-graph over $\Sigma$, 
then
\begin{align}	\label{e:uvzeta1}
	 I_u (t) &\leq  I_v (t) + |\zeta|^2 \, , \\
	 \left( -4\, \pi \, t \right)^{ - \frac{n}{2} } \, \int_{M_t} \frac{|x|^2}{-t} \, v^2 \, \e^{ \frac{|x|^2}{4t}} &\leq  C_n \, \lambda_0 \, |\zeta|^2 +  (n+3) \, 
	\left( -4\, \pi \, t \right)^{ - \frac{n}{2} } \, \int_{M_t} \frac{|x|^2}{-t} \,  u^2    \, \e^{ \frac{|x|^2}{4t}} \, , \\
	-t\, \left| I_{|\nabla u|}(t) - I_{|\nabla v|}(t) \right|  &\leq  (n+2) \, \lambda_0 \, |\zeta|^2  + 2\, |\zeta| \, \left( - \lambda_0 \, t\, I_{u\, |\phi|}(t) \right)^{ \frac{1}{2} }  \, . \label{e:lastuvzeta}
\end{align}\end{Lem}

    \begin{proof}
  We will need   some calculations on $\Sigma$.  Let  $g_{ij}$ denote the matrix 
  of $(4\pi)^{ - \frac{n}{2}} \, \e^{- \frac{|x|^2}{4}}$  Gaussian  inner products of $\{ 1 , x_1 , \dots , x_{n+1} \}$ on $\Sigma$:
    \begin{align}		\label{e:cylij0}
    	\frac{g_{ij}}{\lambda (\SS^k)} &= 
	\begin{cases}
	0 &{\text{ if }} i \ne j \\
	1 &{\text{ if }} i=j=0 \\
	 \frac{2k}{k+1}  &{\text{ if }} i=j \leq k+1 \\
	2  &{\text{ if }} k+1 < i=j \leq n+1 
	\end{cases} 
    \end{align}
   By \eqr{e:cylij0}, $g_{ij}$ is invertible and the largest eigenvalue of $g_{ij}^{-1}$ is $\frac{1}{\lambda (\SS^k)} < \frac{1}{\sqrt{2}}$. Thus, for $R_0$ large and $\epsilon_0$ small,
    the matrix $\bar{g}_{ij}= \bar{g}_{ij}(t)$ of  $J_t$-inner products of $\{ 1 , \frac{x_1}{\sqrt{-t}} , \dots , \frac{x_{n+1}}{\sqrt{-t}} \}$ on $M_t$ is invertible 
    and the largest eigenvalue of $\bar{g}^{-1}_{ij}$ is $< 1$ in norm.
Thus, since $\zeta = \bar{g} (a)$, we have
\begin{align}	\label{e:azeta1}
	a = \bar{g}^{-1} (\zeta) {\text{ and }} |a|^2 = \sum_i a_i^2 \leq |\zeta|^2 \, .
\end{align}
Since $I_u (t) = I_v (t) + I_{u-v}(t)$ and $ I_{u-v}(t) = \sum_{i,j} a_i a_j \bar{g}_{ij} = \langle \zeta , a \rangle \leq |\zeta|^2$, so \eqr{e:azeta1} gives \eqr{e:uvzeta1}.

Using the Cauchy-Schwarz inequality, \eqr{e:azeta1} and the entropy bound for $M_t$ gives
\begin{align}	\label{e:uvzeta2}
	\left( -4\, \pi \, t \right)^{ - \frac{n}{2} } \, \int_{M_t} \frac{|x|^2}{-t} \, v^2 \, \e^{ \frac{|x|^2}{4t}} &\leq (n+3) \, 
	\left( -4\, \pi \, t \right)^{ - \frac{n}{2} } \, \int_{M_t} \frac{|x|^2}{-t} \, \left( u^2 + a_0^2 +  \sum \frac{a_i^2 \, x_i^2}{-t}    \right)  \, \e^{ \frac{|x|^2}{4t}}  \notag \\
	&\leq  C_n \, \lambda_0 \, |\zeta|^2 +  (n+3) \, 
	\left( -4\, \pi \, t \right)^{ - \frac{n}{2} } \, \int_{M_t} \frac{|x|^2}{-t} \,  u^2    \, \e^{ \frac{|x|^2}{4t}}     \, .
\end{align}

To compare the energy of $u$ and $v$, we first write 
\begin{align}	\label{e:nunv1}
	I_{|\nabla v|} - I_{|\nabla u|}   = I_{|\nabla (u-v)|} - 2\, \left( -4\pi \, t \right)^{ - \frac{n}{2} } \, \int_{M_t} \langle \nabla (u-v) , \nabla u \rangle \, \e^{ \frac{|x|^2}{4t} } \, .
\end{align}
We bound the first term on the right using  the Cauchy-Schwarz inequality, \eqr{e:azeta1} and the entropy bound for $M_t$
\begin{align}	\label{e:nunv1a}
	I_{|\nabla (u-v)|} \leq (n+1) \, \sum_i a_i^2 \, I_{ \frac{|\nabla x_i|}{\sqrt{-t}}} \leq   (n+1) \, \frac{\lambda_0}{-t} \, \sum_i a_i^2 \leq  (n+1) \, \frac{\lambda_0}{-t}  \, |\zeta|^2 \, .
\end{align}
Since $\cL_t \, x_i = \frac{x_i}{2t} - \langle \phi , \partial_i \rangle$ by \eqr{e:cLtV} with $V=\partial_i$,  Stokes's theorem and the definition of $\zeta$ give
\begin{align}
	  2\, \left( -4\pi \, t \right)^{ - \frac{n}{2} } \, \int_{M_t} \langle \nabla (u-v) , \nabla u \rangle \, \e^{ \frac{|x|^2}{4t} }  
	&=  -2\, J_t (  u  ,  \cL_t \,  (u-v) ) \notag \\
	&= \sum_{i=1}^{n+1} \frac{a_i \, \zeta_i}{-t} +2\sum_{i=1}^{n+1}  a_i \, J_t (  u  , \frac{\langle \phi , \partial_i \rangle }{\sqrt{-t}} ) \, .
\end{align}
Using this and \eqr{e:nunv1a}   in \eqr{e:nunv1},    the bound \eqr{e:azeta1} and the Cauchy-Schwarz inequality gives
\begin{align}		 
	-t\, \left| I_{|\nabla u|} - I_{|\nabla v|} \right|  \leq  (n+1) \, \lambda_0   \, |\zeta|^2 + 2\, |\zeta|^2  - \lambda_0 \, t\, I_{u\, |\phi|}\, .
\end{align}
\end{proof}

\subsection{Quadratic growth}
Using a  variation of Lemma  \ref{l:atleastlinear}, we will show:
    If $M_t$ is close to a cylinder $\Sigma \subset \RR^{n+1} \subset \RR^N$ and $u$ is a caloric function on $M_t$ that is orthogonal to 
      $\{ 1, \frac{x_1}{\sqrt{-t}} , \dots , \frac{x_{n+1}}{ \sqrt{-t}}\}$, then $u$ grows
      essentially quadratically.   
        The  growth comes from a Poincar\'e inequality on $\Sigma$ for functions  orthogonal to 
   $\{ 1, \frac{x_1}{\sqrt{-t}} , \dots , \frac{x_{n+1}}{ \sqrt{-t}}\}$.
    
  Let  $\kappa$   be the vector in $\RR^{n+2}$ given by  $\kappa_0 =  \left| I_1 (t_1) - I_{1} (t_2)
	\right|$ and  $\kappa_i =  \left|
		  \frac{ I_{x_i}(t_1)}{t_1}  -  \frac{I_{x_i}(t_2)}{t_2}
		 \right|$ for $i=1, \dots , n+1$.   The vector $\kappa$ vanishes when $M_t$ is self-shrinking.
   
     \begin{Lem}		\label{l:atleastquad}
   Given $\mu \in (0, 1/4)$, there exist $\epsilon_{\mu} > 0$,  $R_{\mu} > 0$ and $C_n'$ so that if 
   \begin{itemize}
   \item $B_{R_{\mu}} \cap \frac{ M_t }{\sqrt{-t}}$ is an $\epsilon_{\mu}$ $C^1$-graph over $\Sigma$ for $t_1 \leq t \leq t_2 < 0$,
   \item $u_t = \Delta \, u$, $I_u (t_1)= 1$, and $J_{t_1} (u,1)= J_{t_1}(u , x_i) =   0$ for   $i=1, \dots , (n+1)$,
   \end{itemize}
   then $I_{u} (t_2) \leq \left( \frac{t_1}{t_2} \right)^{2\,\mu-2} + C_n' \,\left( 2 + \mu^{-1} \right)  \, \lambda_0^2 \, \sqrt{|\kappa |} \, \left( \frac{t_1}{t_2} \right)^2 $.  
   \end{Lem}

\begin{proof}
Let $C_n'$ be a large constant to be chosen, depending just on $n$.  Set  
$\omega \equiv   \lambda_0  \, \sqrt{|\kappa |} \, \left(  \frac{t_1}{t_2} \right)^2$.   
We are done if $I_u (t_2) \leq C_n' \, \lambda_0 \, \omega$, so we can assume that
\begin{align}	\label{e:littlec}
	C_n' \, \lambda_0 \,  \omega < I_u (t_2) \leq I_u (t) {\text{ for all }} t \in [t_1 , t_2] \, .
\end{align}
Following Lemma  \ref{l:atleastlinear}, we will show that $I_u$ satisfies a differential inequality that integrates to give the lemma.
Namely, we will show that there exist $\bar{C}_n$ so that  
\begin{align}	\label{e:newdiffin}
	(2-2\mu) \,  I_u (t)  &  \leq 2\, \left( 2 + \mu^{-1} \right) \, \bar{C}_n \, \lambda_0 \, \omega +t \, I_{u}'(t)
\end{align}
This gives that
\begin{align}	\label{e:theguythatholdsT}
	\left( (-t)^{2\mu-2} \, I_u \right)' = (-t)^{2\mu-3} \,  \left\{ (2-2\mu)  \, I_u (t) -t\, I_u'    \right\}  \leq   2\, \left( 2 + \mu^{-1} \right) \, \bar{C}_n \, \lambda_0 \, \omega \, (-t)^{2\mu - 3 }  \, .
\end{align}
Integrating  \eqr{e:theguythatholdsT}  from $t_1$ to $t_2$ gives
\begin{align}
	  (-t_2)^{2\mu - 2} \, I_u(t_2) -  (-t_1)^{2\mu-2} \, I_u (t_1)  \leq   2\, \left( 2 + \mu^{-1} \right) \, \bar{C}_n \, \lambda_0 \, \omega  \, \int_{t_1}^{t_2} (-t)^{2\mu -3} \, dt \, .
\end{align}
The lemma follows from this since $I_u (t_1) = 1$ and $\mu \in (0,1/4)$.

 It remains to prove \eqr{e:newdiffin}.    If $t\, I_u'(t) >   2\,  I_u $, then  \eqr{e:newdiffin} holds.  Hence, suppose  that  
  $t\, I_u'  (t)\leq 2\, I_u (t)$ and, thus,   \eqr{e:Ipri} gives
    \begin{align}	 	\label{e:insteadT}
 	-2\, t\,  I_{|\nabla u|}(t) - t \, I_{u\, |\phi|}(t) = t \, I'_u(t)  \leq 2\, I_u(t)  \, .
 \end{align}
 Combining \eqr{e:insteadT} with   the localization inequality for $u$ in Lemma \ref{l:localizeMt}, 
gives  
 \begin{align}	\label{e:locali2T}
 	 \frac{\left( -4\pi \, t \right)^{ - \frac{n}{2} }}{-t} \int_{M_t} |x|^2 \, u^2 \,  \e^{   \frac{|x|^2}{4t} } & \leq 4n \, I_u(t) - 4t \,  \left( 4 \, I_{|\nabla u|} + I_{u\, |\phi|} \right)
	 \leq (4n+16) \, I_u(t) \, .
 \end{align}
 The Poincar\'e inequality \eqr{e:PoB} for $u$ would imply \eqr{e:newdiffin}  if   $J_t (u,1) = J_t (u, x_i) = 0$ ($u$ has
  localization  by \eqr{e:locali2T}). Since this may not be the case, let $v$ be the  $J_t$-projection  of $u$ orthogonal to 
 $\{ 1, \frac{x_1}{\sqrt{-t}} , \dots , \frac{x_{n+1}}{ \sqrt{-t}}\}$.  We will prove   localization for $v$ to get \eqr{e:PoB}  for $v$ and then deduce
\eqr{e:newdiffin}.

	 Since $ J_{t_1} (u , x_i)=0$, Lemma \ref{l:vV2} gives for any $\epsilon_1 \in (0,1/2)$ and  $t \in [t_1 , t_2]$
	\begin{align}	 		\label{e:e761}
		  \sqrt{-t_2} \, \left|   \frac{J_{t} (u , x_i)}{t}\right|  \leq \frac{5}{2} \, \epsilon_1    + \frac{\kappa_0}{\epsilon_1} 
	  +  \frac{1}{2\, \epsilon_1} \left\{
	\kappa_i+ C_n \, \sqrt{\lambda_0}  \,\left(  \frac{t_1}{t_2} \right)^{ \frac{1}{2}}  \, \kappa_0^{ \frac{1}{2} }	\right\}   \, .  
	\end{align}
	We can assume that $|\kappa| < 1/16$ since   there is nothing to prove if $|\kappa|$ is bounded away from $0$.
	Taking $\epsilon_1 = |\kappa|^{ \frac{1}{4}}$ in \eqr{e:e761} gives
	\begin{align}	\label{e:Jtuxi}
		  \sqrt{-t_2} \, \left|   \frac{J_{t} (u , x_i)}{t}\right|  \leq C_n \sqrt{\lambda_0 }  \left(  \frac{t_1}{t_2} \right)^{ \frac{1}{2}} \, |\kappa|^{ \frac{1}{4}} \, .
	\end{align}
	Since Lemma \ref{l:diffcont}
gives  
  	   $J_{t}^2 (u ,1 )    \leq  \kappa_0$, \eqr{e:Jtuxi} gives a constant $C_n$   so  $  \zeta  $  from \eqr{e:hereiszeta} satisfies
	   \begin{align}		\label{e:zetabound}
	   	|\zeta |^2 \leq C_n \, \lambda_0 \, \left(  \frac{t_1}{t_2} \right)^2 \, \sqrt{|\kappa|} \equiv C_n \, \omega \, .
	   \end{align}
	   Using this in Lemma \ref{l:projction}
 gives  $\bar{C}_n$ so that 
\begin{align}	\label{e:uvzeta1T}
	 I_u (t) &\leq  I_v (t) + \bar{C}_n \, \omega \, , \\
	 \left( -4\, \pi \, t \right)^{ - \frac{n}{2} } \, \int_{M_t} \frac{|x|^2}{-t} \, v^2 \, \e^{ \frac{|x|^2}{4t}} &\leq  \bar{C}_n \, \lambda_0 \, \omega +  (n+3) \, 
	\left( -4\, \pi \, t \right)^{ - \frac{n}{2} } \, \int_{M_t} \frac{|x|^2}{-t} \,  u^2    \, \e^{ \frac{|x|^2}{4t}} \, , \label{e:uvzeta1TS} \\
	-t\, \left| I_{|\nabla u|}(t) - I_{|\nabla v|}(t) \right|  &\leq  \bar{C}_n \, \lambda_0 \,\omega   +  \sqrt{ \bar{C}_n \,\omega} \, \left( - \lambda_0 \, t\, I_{u\, |\phi|}(t) \right)^{ \frac{1}{2} }
	  \, . \label{e:lastuvzetaT}
\end{align}
As long as we choose $C_n' > 2\, \bar{C}_n$, then \eqr{e:littlec} and  \eqr{e:uvzeta1T} guarantee  that
\begin{align}	\label{e:justbeforeTV}
	I_v (t) \leq I_u (t) \leq 2 \, I_v (t) \, .
\end{align}
Similarly, \eqr{e:littlec},   \eqr{e:uvzeta1TS} and \eqr{e:locali2T} give 
\begin{align}	\label{e:IwantmyMTV}
	 \left( -4\, \pi \, t \right)^{ - \frac{n}{2} } \, \int_{M_t} \frac{|x|^2}{-t} \, v^2 \, \e^{ \frac{|x|^2}{4t}} &\leq  I_u(t) +  (n+3) \, (4n+16) \, I_u(t) 
	\, .
\end{align}
Using \eqr{e:insteadT} in \eqr{e:lastuvzetaT} gives
\begin{align}	
	-t\, \left| I_{|\nabla u|}(t) - I_{|\nabla v|}(t) \right|  &\leq  \bar{C}_n \, \lambda_0 \,\omega   + \sqrt{ \bar{C}_n \, \omega} \, \left( 2\, \lambda_0 \, I_u (t) \right)^{ \frac{1}{2} }
	 \, . \label{e:lastuvzetaTT}
\end{align}
By \eqr{e:insteadT}, \eqr{e:justbeforeTV}, \eqr{e:IwantmyMTV}, \eqr{e:lastuvzetaTT},    $v$ satisfies the localization inequality \eqr{e:concinq}.  Therefore, we can apply  \eqr{e:PoB} in Lemma \ref{l:poin1}
to get
\begin{align}	\label{e:PoBT}
	(1-\mu/2) \, I_v(t)  \leq  -t \, I_{|\nabla v|}(t)\, .
\end{align}
Using \eqr{e:uvzeta1T} and \eqr{e:lastuvzetaTT}, \eqr{e:PoBT} implies that
\begin{align}
	  (1-\mu/2) \,  I_u (t)  &\leq  (1-\mu/2) \, I_v(t)  + \bar{C}_n \, \omega  \leq   \bar{C}_n \, \omega  -t \, I_{|\nabla v|}(t) \notag \\
	  &\leq 2\, \bar{C}_n \, \lambda_0 \, \omega     +\sqrt{ \bar{C}_n \, \omega} \,  \left( 2\, \lambda_0 \, I_u (t) \right)^{ \frac{1}{2} } -t \, I_{|\nabla u|}(t)  \, .
\end{align}
Using an absorbing inequality on the middle term gives
\begin{align}	\label{e:finallythere}
		(1-\mu) \,  I_u (t)  &\leq \left( 2 + \mu^{-1} \right) \, \bar{C}_n \, \lambda_0 \, \omega -t \, I_{|\nabla u|}(t) \leq \left( 2 + \mu^{-1} \right) \, \bar{C}_n \, \lambda_0 \, \omega + \frac{t}{2} \, I_{u}'(t) \, .
\end{align}
This gives the desired differential inequality, completing the proof.
 \end{proof}

   \subsection{Sharp bounds for codimension}
     
     The following   implies Theorem \ref{t:ancientcylinder}:
   
   \begin{Thm}	\label{t:ancientP1}
   If    a tangent flow of $M_t$ at $-\infty$  is a cylinder, then $\dim \cP_1(M_t) = n+2$.
   \end{Thm}
   
  By a cylinder, we mean a multiplicity one cylinder.  
The space $\cP_1 (M_t)$ always includes the constant function and the linearly independent coordinate functions, so  $\dim \cP_1 (M_t)$  is 
     $(n+1)$ for an $n$-plane  and at least $(n+2)$ otherwise.  The point is to use the asymptotic cylindrical structure    to prove $\dim \cP_1(M_t) \leq n+2$.
   
       \vskip1mm
We will need   the following uniqueness of blowup type for $M_t$:

   \begin{Lem}	\label{p:ancientcyl}
Suppose that $ \SS^k_{\sqrt{2k}} \times \RR^{n-k}$  is a  tangent flow at $-\infty$.
   Given  $\epsilon > 0$  and $\Lambda > 1$, there exists $T< 0$ so that if $t_0< T$, then there is a rotation $\cR$ of $\RR^N$ so that  
   $B_{\Lambda} \cap \frac{M_t}{\sqrt{-t}}  $ is a graph over $\cR \left( \SS^k_{\sqrt{2k}} \times \RR^{n-k} \right)$ with $C^{2}$ norm at most $\epsilon$ for every $t\in [ \Lambda^2 \, t_0 , t_0]$.
   \end{Lem}
   
   \begin{proof}
This follows from the rigidity of the cylinder of Theorem \ref{t:rigidity} and White's curvature estimate \cite{W1} (cf. corollary $0.3$ in \cite{CIM}).
   \end{proof}

\begin{proof}[Proof of Theorem \ref{t:ancientP1}]
We will get a contradiction if 
 $u_0 \equiv 1, \, u_1 , \dots , u_{n+2}$ are linearly independent functions in $\cP_1 (M_t)$.   
  
Given $\mu > 0$ and $\Omega > 1$,   Lemma \ref{l:goodm} with $d=1$ gives
    $m_{q} \to \infty$ so that $v_1 , \dots , v_{n+2}$ defined by $v_i = \frac{ w_{i,-\Omega^{m_q +1}}}{ \sqrt{f_i (-\Omega^{m_q +1})}}$ satisfy
  \begin{align}	\label{e:goodmPP}
   J_{-\Omega^{m_q+1}} (v_i , v_j) = \delta_{ij}  
   {\text{ and 
}}
 	\sum_{i=1}^{n+2} I_{v_i} (-\Omega^{m_q}) \geq(n+2) \, \Omega^{-1-\mu} \, .
 \end{align}
For $m_q$ sufficiently large,  Lemma \ref{p:ancientcyl} gives that $\frac{M_t}{\sqrt{-t}}$ is as close as we want to a cylinder $\Sigma_q$ (a priori, the cylinder can vary with $q$).  Let $x_1 , \dots , x_{n+1}$ be the coordinate functions for the cylinder $\Sigma_q$.  Make an orthogonal change of basis of    $v_1 , \dots , v_{n+2}$ so that 
\begin{align}	\label{e:vnplus2}
	J_{-\Omega_{m_q+1}} (v_{n+2} , x_i) = 0 {\text{ for }} i = 1, \dots , n+1 \, .
\end{align}
Since trace is invariant under orthogonal changes, \eqr{e:goodmPP} still holds.  

Every $v_i$ is $J_{-\Omega^{m_q+1}}$-orthogonal to the constants for $i\geq 1$.  Thus, given $\mu \in (0, 1/2)$, then for every $m_q$ sufficiently large we can 
apply Lemma \ref{l:atleastlinear}
to get that  
  \begin{align}	\label{e:linear22}
  	I_{v_i} (-\Omega^{m_q}) \leq  \Omega^{\mu-1} + 2 \,  \kappa_{0}  \, .
\end{align}
However, $v_{n+2}$ is also orthogonal to the $x_i$'s and, thus,     the stronger
Lemma \ref{l:atleastquad}
gives
  \begin{align}	\label{e:betternp2}
  	I_{v_{n+2}} (-\Omega^q) \leq \Omega^{2\mu-2} + C_n' \,\left( 2 + \mu^{-1} \right)  \, \lambda_0^2 \, \sqrt{|\kappa |}  \, \Omega^2 \,  .
\end{align}
Using \eqr{e:linear22} for $i\leq n+1$ and \eqr{e:betternp2} for $i=n+2$ gives
  \begin{align}	\label{e:goodmPP1}
 	\sum_{i=1}^{n+2} I_{v_i} (-\Omega^{m_q}) \leq (n+1) \, \Omega^{\mu-1} + 2(n+1) \kappa_0 + \Omega^{2\mu-2} +C_n' \,\left( 2 + \mu^{-1} \right)  \, \lambda_0^2 \, \sqrt{|\kappa |}  \, \Omega^2 \, .
 \end{align}
Combining this with the lower bound \eqr{e:goodmPP} gives
\begin{align}	\label{e:hereIScon}
	(n+2) \, \Omega^{-1-\mu}  \leq  (n+1) \, \Omega^{\mu-1} + 2(n+1) \, |\kappa| + \Omega^{2\mu-2} +C_n' \,\left( 2 + \frac{1}{\mu} \right)  \, \lambda_0^2 \, \sqrt{|\kappa |}  \, \Omega^2   \, .
\end{align}
This gives a contradiction. To see this, fix any $\Omega > 1$ and then choose $\mu > 0$ small so that
\begin{align}	\label{e:heregg}
	(n+2) \, \Omega^{-1-\mu}  > (n+1) \, \Omega^{\mu-1} + \Omega^{2\mu-2}     \, .
\end{align}
Then take $q$ large enough so that $|\kappa|$ is small enough to contradict  \eqr{e:hereIScon}.
  \end{proof}

\end{document}